\documentclass{article}
\usepackage{arxiv}
\usepackage[T1]{fontenc}
\usepackage[english]{babel}
\usepackage{microtype}
\usepackage{amsmath,amssymb,amsthm,mathtools,aliascnt}
\usepackage{booktabs,tabularx,array,multirow}
\usepackage{graphicx,xcolor}
\DeclareRobustCommand{\rev}[1]{#1}
\usepackage{float,needspace}
\usepackage{enumitem}
\usepackage[colorlinks=true,allcolors=blue]{hyperref}
\usepackage[nameinlink,capitalise,noabbrev]{cleveref}
\allowdisplaybreaks
\newtheorem{theorem}{Theorem}[section]
\newaliascnt{lemma}{theorem}
\newtheorem{lemma}[lemma]{Lemma}
\aliascntresetthe{lemma}
\Crefname{lemma}{Lemma}{Lemmas}
\newaliascnt{proposition}{theorem}
\newtheorem{proposition}[proposition]{Proposition}
\aliascntresetthe{proposition}
\Crefname{proposition}{Proposition}{Propositions}
\newaliascnt{corollary}{theorem}

\aliascntresetthe{corollary}
\Crefname{corollary}{Corollary}{Corollaries}
\theoremstyle{definition}
\newaliascnt{assumption}{theorem}
\newtheorem{assumption}[assumption]{Assumption}
\aliascntresetthe{assumption}
\Crefname{assumption}{Assumption}{Assumptions}
\newaliascnt{example}{theorem}
\newtheorem{example}[example]{Example}
\aliascntresetthe{example}
\Crefname{example}{Example}{Examples}
\theoremstyle{remark}
\newaliascnt{remark}{theorem}
\newtheorem{remark}[remark]{Remark}
\aliascntresetthe{remark}
\Crefname{remark}{Remark}{Remarks}

\numberwithin{equation}{section}
\hypersetup{
    pdftitle={Numerical methods for nonlinear stochastic differential equations driven by subdiffusion}
}

\title{Numerical methods for nonlinear stochastic differential equations driven by subdiffusion}
\author{Qian Guo \quad Wei Liu \quad Ruchun Zuo\thanks{Corresponding author. Email: \href{mailto:ruchunzuo@outlook.com}{ruchunzuo@outlook.com}.}\\[0.5em]
{\small Department of Mathematics, Shanghai Normal University}\\
{\small Shanghai 200234, China}}
\date{}

\begin{document}
	\maketitle
	
\begin{abstract}
Stochastic differential equations (SDEs) driven by subdiffusion see great applications in anomalous transport and asset price dynamics.
Trusted numerical methods are essential for their application in practice. 
In this paper, for SDEs driven by subdiffusion with superlinear coefficients, we propose a brand-new explicit numerical method. 
The sharp convergence rate is proved with new techniques. 
The potential application of the proposed method in numerically solving a generalized Fokker--Planck equation is also illustrated. 
Numerical simulations are provided to demonstrate the theoretical results.
\end{abstract}



\section{Introduction}\label{sec:intro}

For stochastic differential equations (SDEs) driven by subdiffusion, numerical methods based on subordinator-grid have significant advantages on the computational cost, as they avoid the costly approximation of the inverse subordinator. 
However,  a long-standing limitation of such methods is their reliance on the assumption that drifts integrated against the time variable should not depend on the state. 
Therefore, it is a major challenge whether this type of methods can be extended to the case of state-dependent drifts.
In this paper, we give a positive answer by developing an explicit numerical framework that handles state-dependent dynamics driven by both the time variable and the subdiffusion process.
The proposed scheme preserves the defining efficiency of subordinator-grid algorithms and achieves a sharp strong convergence order.

In applications, anomalous subdiffusion arises in systems with trapping and heavy-tailed waiting times and can be modeled by scaling limits of continuous-time random walks~\cite{MetzlerKlafter2000,MeerschaertScheffler2004}.
Stochastic representations of fractional diffusion and Fokker--Planck equations (FPEs) have been established for both stable and more general infinitely divisible waiting times~\cite{MagdziarzWeronWeron2007,MagdziarzWeronKlafter2008,Magdziarz2009Langevin}, and pathwise simulations offer an effective probabilistic tool for computing their solutions~\cite{CarnaffanKawai2017}.
In particular, subdiffusion models with external forces have been studied in which the drift is integrated with respect to the time variable $t$ or the inverse subordinator $E(t)$~\cite{EuleFriedrich2009,CairoliBaule2015,ChenWangDeng2019}.
Both types of state-dependent drift appear together in the subdiffusive Black--Scholes model proposed in~\cite{zhang2025subdiffusive}.
Motivated by these applications, in this paper we consider nonautonomous SDEs driven by subdiffusion of the form
\begin{equation}\label{eq:intro-model}
	\mathrm{d}X(t) = h(t, E(t), X(t)) \mathrm{d}t + f(t, E(t), X(t)) \mathrm{d}E(t) + g(t, E(t), X(t)) \mathrm{d}B(E(t)),
\end{equation}
where $B$ is a standard Brownian motion, $E$ is the inverse of a subordinator $D$ independent of $B$, and the coefficients satisfy appropriate monotonicity and growth conditions.
Here subdiffusion refers to the time-changed Brownian motion $B(E(t))$.

Now, we do a brief review of exsiting literatures about techniques dealing with the super-linearity and numerical methods for SDE driven subdiffusion.  

Because explicit solutions to the nonlinear SDE~\eqref{eq:intro-model} are rarely available, numerical schemes for approximating its sample paths are essential.
Even when closed-form solutions exist, path simulations remain useful to generate sample trajectories for statistical analysis and hypothesis testing.
However, for classical SDEs, strong convergence beyond global Lipschitz conditions was established under suitable moment bounds in~\cite{HighamMaoStuart2002}, whereas the standard explicit Euler--Maruyama (EM) scheme is known to diverge in finite time under superlinear growth~\cite{HutzenthalerJentzenKloeden2011}.
To preserve the low computational cost of explicit time-stepping, various modified explicit schemes have been developed, including tamed Euler methods~\cite{HutzenthalerJentzenKloeden2012,Sabanis2016} and truncated EM methods~\cite{Mao2015,Mao2016,LiMaoYin2019}.

Beyond the challenge of superlinear growth, simulating these SDEs requires resolving the distinct dynamics driven by $\mathrm{d}t$ and $\mathrm{d}E(t)$.
To address these difficulties, existing numerical schemes for SDEs driven by subdiffusion fall into three main strategies.
The first strategy applies the duality principle~\cite{Kobayashi2011} to SDEs driven by subdiffusion of the form
\begin{equation*}
\mathrm{d}X(t)=f(E(t),X(t)) \mathrm{d}E(t)+g(E(t),X(t)) \mathrm{d}B(E(t)).
\end{equation*}
This principle represents the solution as a classical SDE solution evaluated at $E(t)$.
In this direction, strong and weak convergence of the EM scheme under global Lipschitz conditions was established in~\cite{JumKobayashi2016}, and a semi-implicit EM method for SDEs with superlinear drift and H\"older continuous time dependence was developed in~\cite{DengLiu2020}.
More recently, strong convergence of order $1/2$ in $L^2$ for backward and projected Euler methods with superlinear drift and diffusion coefficients was obtained in~\cite{WangNiuZhang2026}.
For time-changed L\'evy noise, strong and weak convergence rates of numerical methods were also studied in~\cite{ChenLiuWu2026}.
The advantage of this strategy is that classical SDE convergence results transfer directly through the time change.
However, the standard duality principle does not directly apply to general $t$-dependent coefficients or equations with a drift integrated against $t$.

The second strategy constructs numerical schemes on the random subordinator grid $\tau_i = D(i\delta)$ ($D$-grid), where $\delta > 0$ is the sampling step size and $E(\tau_i) = i\delta$.
For SDEs driven by subdiffusion of the form
\begin{equation*}
\mathrm{d}X(t) = f(t, X(t)) \mathrm{d}E(t) + g(t, X(t)) \mathrm{d}B(E(t)),
\end{equation*}
an EM scheme was constructed directly on the $D$-grid under global Lipschitz assumptions in~\cite{JinKobayashi2019}.
This construction does not rely on the duality principle.
For SDEs with superlinear coefficients and no drift integrated against $t$, truncated EM schemes were developed on the $D$-grid in~\cite{LiuMaoTangWu2020,LiLiuTang2025}, and a Milstein-type approximation was proposed in~\cite{WuZuo2026}.
More closely related to the present model, a $D$-grid EM scheme was studied in~\cite{JinKobayashi2021} for SDEs with a drift integrated against $t$ and another against $E$ of the form
\begin{equation*}
\mathrm{d}X(t) = h(E(t)) \mathrm{d}t + f(E(t), X(t)) \mathrm{d}E(t) + g(E(t), X(t)) \mathrm{d}B(E(t)),
\end{equation*}
and shown to converge strongly with order $1/2$ with respect to the subordinator step size $\delta$.
The main advantage of the $D$-grid is that it avoids inverse-subordinator evaluations and permits the direct simulation of Brownian increments over intervals of length $\delta$.
Furthermore, for a related McKean--Vlasov SDE driven by subdiffusion whose drift integrated against $t$ also depends only on $E(t)$, strong approximation using interacting particle systems was studied in~\cite{WenLiXu2023}.
However, as noted in~\cite{JinKobayashi2021}, the analysis does not directly extend when $h$ depends on $X(t)$, because reducing the subordinator step size $\delta$ does not force the lengths of the random intervals $[\tau_i, \tau_{i+1})$ generated by subordinator jumps to vanish.
Indeed, a counterexample in~\cite{ZuoGuo2026TCSDDE} shows that even for the linear equation $\mathrm{d}X(t) = X(t) \mathrm{d}t$, the pure $D$-grid approximation fails to converge as $\delta \to 0$.

The third strategy constructs numerical schemes on an equidistant time grid $t_n = n\Delta$.
For SDEs driven by subdiffusion generated by an inverse $\alpha$-stable subordinator with $\alpha\in(0,1)$ and without a drift integrated against $t$, an $L^1$ convergence rate of $\alpha$ was established for the backward EM scheme with additive noise in~\cite{LiuWangZuo2025}.
Under multiplicative noise, rates close to $\alpha/2$ were obtained in~\cite{Zuo2026} for the standard and truncated EM schemes.
Building on this equidistant framework, an equidistant truncated EM (ETEM) scheme was proposed in~\cite{ZuoGuo2026} for nonautonomous superlinear SDEs with a state-dependent drift integrated against $t$, as in~\eqref{eq:intro-model}.
Its main advantage is that every time step has deterministic length $\Delta$, so the state-dependent drift integrated against $t$ is resolved directly, yielding a strong convergence rate determined by $\alpha/2$ and the temporal H\"older exponents of the coefficients.
Nevertheless, this equidistant approach has two main limitations.
First, the schemes in~\cite{LiuWangZuo2025,Zuo2026,ZuoGuo2026} require values of the inverse subordinator $E(t_n)$ at the uniform time-grid points.
In a fully discrete implementation, these values must be approximated from samples of $D$ with step size $\delta$.
The resulting approximation error was analyzed in~\cite{ZuoGuo2026TCSDDE} for superlinear stochastic delay differential equations driven by subdiffusion.
Second, for the equidistant schemes considered above, the strong convergence rates in $\Delta$ depend on the index $\alpha$ and the temporal regularity of the coefficients.
Indeed, for the fully discrete ETEM scheme studied in~\cite{ZuoGuo2026TCSDDE}, a multiplicative-noise example yields a lower bound of order $\Delta^{\alpha/2}+\delta^{1/2}$, establishing the sharpness of these polynomial exponents over the class of equations considered there.

Building on the second strategy, we develop an explicit numerical method for nonautonomous superlinear SDEs driven by subdiffusion with a state-dependent drift integrated against $t$.
Compared with pure $D$-grid schemes and the equidistant methods in~\cite{ZuoGuo2026,ZuoGuo2026TCSDDE}, the proposed method offers two main advantages.
First, it completely avoids evaluating or approximating the inverse subordinator $E(t_n)$.
Second, for general subordinators, the method achieves a strong convergence order of $1/2$ in both discretization parameters when the coefficients are at least $1/2$-H\"older continuous in both time arguments, compared with the $\alpha/2$ rate discussed above for equidistant schemes driven by inverse $\alpha$-stable subordinators.
For coefficients with lower temporal regularity, explicit sharpened rates are established via the Laplace exponent.

The remainder of the paper is organized as follows.
Section~\ref{sec:model} introduces the assumptions and useful lemmas.
Section~\ref{sec:method} constructs the numerical method and establishes its moment bounds.
Section~\ref{sec:convergence} proves the strong convergence rates for general subordinators and the sharpened rates under a growth condition on the Laplace exponent.
Section~\ref{sec:numerical} presents numerical experiments that confirm the theoretical convergence rates and demonstrate the application of the method to a generalized FPE.
Section~\ref{sec:conclusion} concludes the paper.

\section{Notations and properties of the underlying equation}\label{sec:model}

Let $T>0$, $d,m\in\mathbb N$, and $x_0\in\mathbb R^d$ be deterministic.
Let $(\Omega_B,\mathcal F_B,\mathbb P_B)$ and $(\Omega_D,\mathcal F_D,\mathbb P_D)$ carry an $m$-dimensional standard Brownian motion $B$ and a subordinator $D$, respectively.
We work on the completed product probability space $(\Omega,\mathcal F,\mathbb P)$ with $\mathbb P=\mathbb P_B\otimes\mathbb P_D$, so $B$ and $D$ are independent.
We write $|\cdot|$ for the Euclidean norm, $\|\cdot\|$ for the Frobenius norm, and $\langle\cdot,\cdot\rangle$ for the Euclidean inner product.
The superscript $\top$ denotes transposition.
In estimates stated simultaneously for vector- and matrix-valued coefficients, $|\cdot|$ denotes the corresponding norm.
The notation $\mathbb E_B$ denotes integration over the Brownian coordinate with the path of $D$ fixed, and $\mathbb E_D$ denotes integration over the subordinator coordinate.
Thus $\mathbb E[\cdot]=\mathbb E_D[\mathbb E_B[\cdot]]$.
	
The subordinator $D=(D(u))_{u\ge0}$ has zero drift, no killing, $D(0)=0$, and infinite L\'evy measure $\nu(0,\infty)=\infty$.
The Laplace transform of $D$ is given by
\begin{equation*}
	\mathbb E[e^{-\lambda D(u)}] = e^{-u\Phi(\lambda)}, \qquad \Phi(\lambda) = \int_{(0,\infty)} (1-e^{-\lambda z}) \nu(\mathrm dz),
\end{equation*}
where the Laplace exponent $\Phi(\lambda)\to\infty$ as $\lambda\to\infty$.
The inverse subordinator $E = (E(t))_{t \ge 0}$ is defined as the first passage time
\begin{equation*}
	E(t) = \inf\{u \ge 0 : D(u) > t\}, \qquad t \ge 0.
\end{equation*}
Because $D$ is strictly increasing and unbounded almost surely, $E$ has continuous and nondecreasing sample paths with $E(0)=0$.
A standard example is the $\alpha$-stable subordinator with $\Phi(\lambda) = \lambda^\alpha$, $0 < \alpha < 1$.
Properties of its inverse and connections with time-fractional equations are reviewed in \cite{MeerschaertStraka2013}.
Tempered stable and Gamma subordinators are also covered by our analysis whenever their sampled increments are available.
For properties of inverse tempered stable and inverse Gamma subordinators, see \cite{KumarVellaisamy2015,ColantoniDOvidio2023}.

For martingale properties and stochastic calculus for Brownian motion evaluated at inverse subordinators, see \cite{Kobayashi2011,MagdziarzSchilling2015}.
Under $\mathbb P_B$, $E$ is deterministic and $B(E(t))$ is a continuous martingale with quadratic covariation matrix $E(t)I_m$.
Adaptedness and predictability are understood with respect to the usual augmentation of the Brownian history up to $E(t)$, with the path of $D$ fixed.
To combine the $\mathrm dt$ and $\mathrm dE(t)$ contributions, define the strictly increasing function $A(t)=t+E(t)$.

The SDE~\eqref{eq:intro-model} is written in integral form as
\begin{equation}\label{eq:model}
	X(t) = x_0 + \int_0^t h(s, E(s), X(s)) \mathrm ds + \int_0^t f(s, E(s), X(s)) \mathrm dE(s) + \int_0^t g(s, E(s), X(s)) \mathrm dB(E(s)),
\end{equation}
for $t \in [0, T]$, where the drift coefficients $h, f\colon [0, T] \times [0, \infty) \times \mathbb R^d \to \mathbb R^d$ and the diffusion coefficient $g\colon [0, T] \times [0, \infty) \times \mathbb R^d \to \mathbb R^{d \times m}$ are Borel measurable functions.
Here $h$ is the coefficient of the drift integrated against $t$, and $f$ is the coefficient of the drift integrated against $E$.
The coefficients may exhibit superlinear growth with respect to the state variable.
We impose the following structural conditions, where all constants are uniform in $t \in [0, T]$ and $u \in [0, \infty)$.

\begin{assumption}\label{ass:space}
	There exist constants $L_1 > 0$ and $\ell \ge 0$ such that, for all $\varphi \in \{h, f, g\}$, $t \in [0, T]$, $u \ge 0$, and $x, y \in \mathbb R^d$,
	\begin{equation*}
		|\varphi(t, u, x) - \varphi(t, u, y)| \le L_1 (1 + |x|^\ell + |y|^\ell) |x - y|, \qquad |\varphi(t, u, 0)| \le L_1.
	\end{equation*}
	In particular, $|\varphi(t, u, x)| \le C(1 + |x|^{\ell+1})$ for all $(t, u, x)$.
\end{assumption}

\begin{assumption}\label{ass:time}
	There exist constants $L_2 > 0$ and regularity exponents $\beta_{\varphi,t}, \beta_{\varphi,u} \in (0, 1]$, for $\varphi \in \{h, f, g\}$, such that for all $s, t \in [0, T]$, $u, v \ge 0$, and $x \in \mathbb R^d$,
	\begin{equation*}
		|\varphi(t, u, x) - \varphi(s, v, x)| \le L_2 (1 + |x|^{\ell+1}) \bigl(|t - s|^{\beta_{\varphi,t}} + |u - v|^{\beta_{\varphi,u}}\bigr).
	\end{equation*}
	For notational convenience, we set $\beta_t = \min_{\varphi \in \{h,f,g\}} \beta_{\varphi,t}$ and $\beta_u = \min_{\varphi \in \{h,f,g\}} \beta_{\varphi,u}$.
\end{assumption}

\begin{assumption}\label{ass:mono}
	There exist constants $L_3 > 0$ and $q > 2$ such that for all $t \in [0, T]$, $u \ge 0$, and $x, y \in \mathbb R^d$,
	\begin{align*}
		2\langle x - y, h(t, u, x) - h(t, u, y)\rangle &\le L_3 |x - y|^2, \\
		2\langle x - y, f(t, u, x) - f(t, u, y)\rangle + (q - 1) \|g(t, u, x) - g(t, u, y)\|^2 &\le L_3 |x - y|^2.
	\end{align*}
\end{assumption}

\begin{assumption}\label{ass:coerc}
	There exist a finite real exponent $P > 2$ and a constant $L_4 > 0$ such that for all $t \in [0, T]$, $u \ge 0$, and $x \in \mathbb R^d$,
	\begin{align*}
		2\langle x, h(t, u, x)\rangle &\le L_4 (1 + |x|^2), \\
		2\langle x, f(t, u, x)\rangle + (P - 1) \|g(t, u, x)\|^2 &\le L_4 (1 + |x|^2).
	\end{align*}
\end{assumption}

Assumption~\ref{ass:coerc} requires coercivity up to a finite order $P$.
Existence of moments of all orders is not assumed.
The numerical moment and convergence estimates below impose additional lower bounds on $P$.
Throughout this paper, $C$ denotes a generic positive constant that may depend on the structural parameters, the chosen moment orders, $T,d,m,x_0$, and the law of $D$, but is independent of the discretization parameters.

General moment formulas and tail estimates for inverse subordinators are given in \cite{Lageras2005}.
The following exponential integrability estimate follows from \cite{JumKobayashi2016} by absorbing polynomial factors into the exponential.

\begin{lemma}\label{lem:inverse-moments}
	For every $c > 0$ and $a \ge 0$, $\mathbb E[(1 + E(T))^a e^{c E(T)}] < \infty$.
\end{lemma}

For related proofs of well-posedness and moment estimates for the exact solution, see \cite{ZuoGuo2026,ZuoGuo2026TCSDDE}.

\begin{proposition}\label{prop:exact}
	Under Assumptions~\ref{ass:space}, \ref{ass:time}, and~\ref{ass:coerc}, the SDE~\eqref{eq:model} admits a unique continuous nonexplosive strong solution $X(t)$ on $[0, T]$.
	Moreover, for every $0 < r < P$,
	\begin{equation*}
		\mathbb E_B\left[\sup_{0 \le t \le T} (1 + |X(t)|^2)^{\frac{r}{2}}\right] \le C e^{C(T + E(T))}.
	\end{equation*}
	Consequently, $\mathbb E[\sup_{0 \le t \le T} |X(t)|^r] < \infty$.
\end{proposition}
	
To evaluate the convergence rate of numerical approximations, we establish a stability estimate.
Suppose an adapted continuous process $\bar{X}$ satisfies the perturbed SDE
\begin{align*}
	\bar{X}(t) &= x_0 + \int_0^t \bigl(h(s, E(s), \bar{X}(s)) - \mathcal R_h(s)\bigr) \mathrm ds + \int_0^t \bigl(f(s, E(s), \bar{X}(s)) - \mathcal R_f(s)\bigr) \mathrm dE(s) \\
	&\quad + \int_0^t \bigl(g(s, E(s), \bar{X}(s)) - \mathcal R_g(s)\bigr) \mathrm dB(E(s)),
\end{align*}
driven by predictable residual processes $\mathcal R_h, \mathcal R_f$, and $\mathcal R_g$.
For $p \ge 2$, set
\begin{equation}\label{eq:forcing}
	\mathcal I(t) = \int_0^t \mathbb E_B\left[|\mathcal R_h(s)|^p\right]^{\frac{2}{p}} \mathrm ds + \int_0^t \left(\mathbb E_B\left[|\mathcal R_f(s)|^p\right]^{\frac{2}{p}} + \mathbb E_B\left[\|\mathcal R_g(s)\|^p\right]^{\frac{2}{p}}\right) \mathrm dE(s).
\end{equation}

\begin{lemma}\label{lem:stability}
	Under Assumptions~\ref{ass:space}--\ref{ass:coerc}, let $2 \le p < q$.
	If $\mathcal I(T) < \infty$ almost surely, then
	\begin{equation*}
		\mathbb E_B\left[\sup_{0 \le t \le T} |X(t) - \bar{X}(t)|^p\right] \le C e^{C(T+E(T))} \mathcal I(T)^{\frac{p}{2}}.
	\end{equation*}
\end{lemma}
	
	\begin{proof}
For each $n > |x_0|$, introduce the stopping time
\begin{equation*}
	\sigma_n = \inf\{t \in [0, T] : |X(t)| \vee |\bar{X}(t)| \ge n\}.
\end{equation*}
To evaluate the deviation between $X$ and $\bar{X}$, set $e(t) = X(t) - \bar{X}(t)$, and for each $\varphi \in \{h, f, g\}$ write
\begin{equation*}
	\Delta \varphi(s) = \varphi(s, E(s), X(s)) - \varphi(s, E(s), \bar{X}(s)).
\end{equation*}
Applying It\^o's formula to the stopped error $|e(t \wedge \sigma_n)|^p$, we obtain
\begin{align}
	|e(t \wedge \sigma_n)|^p
	&\le \int_0^t \mathbf 1_{\{s \le \sigma_n\}} p |e(s)|^{p-2} \langle e(s), \Delta h(s) + \mathcal R_h(s)\rangle \mathrm ds \nonumber\\
	&\quad + \int_0^t \mathbf 1_{\{s \le \sigma_n\}} p |e(s)|^{p-2} \langle e(s), \Delta f(s) + \mathcal R_f(s)\rangle \mathrm dE(s) \nonumber\\
	&\quad + \frac{p(p-1)}{2} \int_0^t \mathbf 1_{\{s \le \sigma_n\}} |e(s)|^{p-2} \|\Delta g(s) + \mathcal R_g(s)\|^2 \mathrm dE(s) + M(t). \label{eq:ito-error}
\end{align}
where the local martingale term is given by
\begin{equation*}
	M(t) = p\int_0^t \mathbf 1_{\{s\le\sigma_n\}} |e(s)|^{p-2} e(s)^\top \bigl(\Delta g(s) + \mathcal R_g(s)\bigr) \mathrm dB(E(s)).
\end{equation*}
Splitting the diffusion perturbation with parameter $\varepsilon = (q - p)/(2(p - 1))$, we have
\begin{equation*}
	\|\Delta g(s) + \mathcal R_g(s)\|^2 \le (1 + \varepsilon)\|\Delta g(s)\|^2 + \Bigl(1 + \frac{1}{\varepsilon}\Bigr)\|\mathcal R_g(s)\|^2.
\end{equation*}
Using this bound, combining the drift contribution involving $\Delta f$ and the diffusion contribution involving $\Delta g$ via Assumption~\ref{ass:mono} gives
\begin{align*}
	&\quad p |e(s)|^{p-2}\Bigl(\langle e(s), \Delta f(s)\rangle + \frac{(p-1)(1+\varepsilon)}{2}\|\Delta g(s)\|^2\Bigr) \\
	&= p |e(s)|^{p-2}\Bigl(\langle e(s), \Delta f(s)\rangle + \frac{q-1}{2}\|\Delta g(s)\|^2 - \frac{q-p}{4}\|\Delta g(s)\|^2\Bigr) \\
	&\le \frac{p L_3}{2}|e(s)|^p - \frac{p(q-p)}{4}|e(s)|^{p-2}\|\Delta g(s)\|^2.
\end{align*}
Similarly, $\langle e(s), \Delta h(s)\rangle \le \frac{L_3}{2}|e(s)|^2$ by Assumption~\ref{ass:mono}.
Applying Young's inequality to the drift residuals for $\varphi \in \{h, f\}$ and inserting these estimates into~\eqref{eq:ito-error}, we obtain
\begin{align}
	&\quad |e(t\wedge\sigma_n)|^p+\frac{p(q-p)}4\int_0^t\mathbf 1_{\{s\le\sigma_n\}}|e(s)|^{p-2}\|\Delta g(s)\|^2 \mathrm dE(s) \nonumber\\
	&\le \frac{p(L_3 + 1)}{2}\int_0^t\mathbf 1_{\{s\le\sigma_n\}}|e(s)|^p \mathrm dA(s)+\frac{p}{2}\int_0^t\mathbf 1_{\{s\le\sigma_n\}}|e(s)|^{p-2}|\mathcal R_h(s)|^2 \mathrm ds \nonumber\\
	&\quad+\int_0^t\mathbf 1_{\{s\le\sigma_n\}}|e(s)|^{p-2}\left(\frac{p}{2}|\mathcal R_f(s)|^2+\frac{p(p-1)}{2}\Bigl(1+\frac{1}{\varepsilon}\Bigr)\|\mathcal R_g(s)\|^2\right) \mathrm dE(s)+M(t). \label{eq:energy}
\end{align}
Taking $\mathbb E_B$ on both sides of~\eqref{eq:energy}, noting $\mathbb E_B[M(t)] = 0$, and dropping the second nonnegative term on the left-hand side, we obtain
\begin{align*}
	&\quad\mathbb E_B\left[|e(t \wedge \sigma_n)|^p\right]\\
	&\le C \int_0^t \mathbb E_B\left[|e(s \wedge \sigma_n)|^p\right] \mathrm dA(s) + C \sum_{\varphi \in \{h, f, g\}} \int_0^t \left(\mathbb E_B\left[|e(s \wedge \sigma_n)|^p\right]\right)^{\frac{p-2}{p}} \left(\mathbb E_B\left[\|\mathcal R_\varphi(s)\|^p\right]\right)^{\frac{2}{p}} \mathrm d\mu_\varphi(s) \\
	&\le C \int_0^t \mathbb E_B\left[|e(s \wedge \sigma_n)|^p\right] \mathrm dA(s) + C \int_0^t \mathbb E_B\left[|e(s \wedge \sigma_n)|^p\right]^{1-\frac{2}{p}} \mathrm d\mathcal I(s) \\
	&\le C \int_0^t \mathbb E_B\left[|e(s \wedge \sigma_n)|^p\right] \mathrm dA(s) + C \mathcal I(t) \left(\sup_{0 \le r \le t} \mathbb E_B\left[|e(r \wedge \sigma_n)|^p\right]\right)^{1-\frac{2}{p}},
\end{align*}
where $\mathrm d\mu_h(s) = \mathrm ds$, $\mathrm d\mu_f(s) = \mathrm d\mu_g(s) = \mathrm dE(s)$, and we used the H\"older inequality and invoked the definition of $\mathcal I(t)$ in~\eqref{eq:forcing}.
Taking the supremum over $u \in [0, t]$ and applying the Gronwall inequality gives
\begin{align}\label{eq:gronwall-bound}
	\sup_{0 \le u \le t} \mathbb E_B\left[|e(u \wedge \sigma_n)|^p\right] &\le C e^{C A(t)} \mathcal I(t) \left(\sup_{0 \le u \le t} \mathbb E_B\left[|e(u \wedge \sigma_n)|^p\right]\right)^{1-\frac{2}{p}} \nonumber\\
	&\le C e^{C A(T)} \mathcal I(T)^{\frac{p}{2}},
\end{align}
where the second inequality follows by dividing both sides by the supremum factor.
Conversely, dropping the first term on the left-hand side of~\eqref{eq:energy} at $t = T$ and substituting~\eqref{eq:gronwall-bound} into the right-hand side, we similarly obtain
\begin{equation}\label{eq:diffusion-energy}
	\mathbb E_B\left[\int_0^T \mathbf 1_{\{s \le \sigma_n\}} |e(s)|^{p-2} \|\Delta g(s)\|^2 \mathrm dE(s)\right] \le C e^{C A(T)} \mathcal I(T)^{\frac{p}{2}}.
\end{equation}

For the local martingale term, the quadratic variation satisfies
\begin{equation*}
	[M, M]_T \le p^2\left(\sup_{0 \le t \le T} |e(t \wedge \sigma_n)|^p\right) \int_0^T \mathbf 1_{\{s \le \sigma_n\}} |e(s)|^{p-2} \|\Delta g(s) + \mathcal R_g(s)\|^2 \mathrm dE(s).
\end{equation*}
By the Burkholder--Davis--Gundy and Young inequalities, and invoking~\eqref{eq:gronwall-bound} and~\eqref{eq:diffusion-energy},
\begin{align}
	&\quad\mathbb E_B\left[\sup_{0 \le t \le T} |M(t)|\right]\nonumber\\
	&\le C \mathbb E_B\left[[M, M]_T^{\frac{1}{2}}\right] \nonumber\\
	&\le \frac{1}{2} \mathbb E_B\left[\sup_{0 \le t \le T} |e(t \wedge \sigma_n)|^p\right] + C \mathbb E_B\left[\int_0^T \mathbf 1_{\{s \le \sigma_n\}} |e(s)|^{p-2} \left(\|\Delta g(s)\|^2 + \|\mathcal R_g(s)\|^2\right) \mathrm dE(s)\right] \nonumber\\
	&\le \frac{1}{2} \mathbb E_B\left[\sup_{0 \le t \le T} |e(t \wedge \sigma_n)|^p\right] + C e^{C A(T)} \mathcal I(T)^{\frac{p}{2}}. \label{eq:martingale-bdg}
\end{align}
Taking the supremum over $t \in [0, T]$ and the expectation $\mathbb E_B$ on both sides of~\eqref{eq:energy}, and using~\eqref{eq:gronwall-bound} and~\eqref{eq:martingale-bdg} to absorb the martingale term, we obtain
\begin{equation}\label{eq:stopped-error-bound}
	\mathbb E_B\left[\sup_{0 \le t \le T} |e(t \wedge \sigma_n)|^p\right] \le C e^{C A(T)} \mathcal I(T)^{\frac{p}{2}}.
\end{equation}
Since the sample paths of $X$ and $\bar{X}$ are almost surely bounded on $[0, T]$, the stopping times satisfy $\sigma_n \to \infty$ almost surely as $n \to \infty$.
Applying Fatou's lemma to~\eqref{eq:stopped-error-bound} and recalling $A(T) = T + E(T)$ completes the proof.
\end{proof}

\section{Numerical method and its moment bounds}\label{sec:method}

The standard $D$-grid EM discretization can fail for SDEs with a state-dependent drift integrated against $t$, as demonstrated in~\cite{ZuoGuo2026TCSDDE}.
Large subordinator jumps leave intervals on which the drift integrated against $t$ is evaluated at only one state, which can prevent convergence.
To address these difficulties, we develop a sub-stepped truncated Euler--Maruyama (SSTEM) scheme.
It uses substeps of length at most $\Delta$ to resolve the drift integrated against $t$ and applies truncation to control the superlinear coefficients, following the approach in~\cite{Mao2015,Mao2016}.
The drift integrated against $E$ and the diffusion driven by $B(E)$ are updated at completed sampled endpoints using increments of length $\delta$ in the argument of $B$.

\subsection{Construction of the scheme}\label{subsec:stem-scheme}

Let $\Delta\in(0,1]$ be the maximum substep size for the drift integrated against $t$ and $\delta\in(0,1]$ the subordinator sampling step size, and set the combined mesh parameter
\begin{equation*}
	\rho=\max\{\Delta,\delta\}.
\end{equation*}
For fixed parameters $\kappa\in(0,1/4]$ and $R_0\ge1$, define the truncation radius
\begin{equation*}
	R=R_0\rho^{-\frac{\kappa}{\ell+1}},
\end{equation*}
and the radial projection $\pi_\rho\colon \mathbb R^d \to \mathbb R^d$ by
\begin{equation}\label{eq:truncation}
	\pi_\rho(x) = \begin{cases} x, & |x| \le R, \\ \frac{R x}{|x|}, & |x| > R. \end{cases}
\end{equation}
The truncated coefficients are $\varphi_\rho(t,u,x)=\varphi(t,u,\pi_\rho(x))$ for $\varphi\in\{h,f,g\}$.

The construction assumes that increments of $D$ with the prescribed law can be sampled.
For simulation methods for stable and tempered stable subordinator increments, see \cite{Magdziarz2009TimeDependent,GrabchakSaba2025}.
Errors introduced by an additional approximation of these increments are not included in the estimates below.

Sample $D(i\delta)$ until the first value exceeding $T$, and write $\tau_i=D(i\delta)$ and $N=\max\{i:\tau_i\le T\}$.
By the definition of $\tau_i$ and the monotonicity of $E$, we have
\begin{equation*}
	0 \le E(s)-i\delta\le\delta, \qquad \tau_i\le s\le\tau_{i+1}\wedge T.
\end{equation*}
The scheme uses these samples and the Brownian values $B(i\delta)$, with no exact evaluations of $E$.
Let $X_i$ approximate $X(\tau_i)$, starting from $X_0=x_0$, and write $X_{\Delta,\delta}(\tau_i)=X_i$.

Write $\Delta\tau_i=(\tau_{i+1}\wedge T)-\tau_i$ for $i=0,\ldots,N$.
On each interval with $\tau_i<T$, take $m_i=\lceil\Delta\tau_i/\Delta\rceil$ substeps for the drift integrated against $t$, with nodes
\begin{equation*}
	t_i^j=(\tau_i+j\Delta)\wedge\tau_{i+1}\wedge T, \qquad j=0,\ldots,m_i.
\end{equation*}
For notational convenience in what follows, we denote the substep increments by $\Delta t_i^j = t_i^{j+1}-t_i^j$.
Clearly, $\Delta t_i^j\le\Delta$ and $\sum_{j=0}^{m_i-1}\Delta t_i^j=\Delta\tau_i$.
Let $X_i^j$ denote the stage value at $t_i^j$, with $X_i^0=X_i$.
The update for the drift integrated against $t$ is
\begin{equation}\label{eq:driftstages}
	X_i^{j+1}=X_i^j+\Delta t_i^j h_\rho(t_i^j,i\delta,X_i^j), \qquad j=0,\ldots,m_i-1.
\end{equation}
If $i<N$, add the increments of the drift integrated against $E$ and the diffusion driven by $B(E)$ at the right endpoint:
\begin{equation}\label{eq:blockupdate}
	X_{i+1} = X_i^{m_i}+\delta f_\rho(t_i^{m_i-1},i\delta,X_i^{m_i-1}) +g_\rho(t_i^{m_i-1},i\delta,X_i^{m_i-1})\Delta B_i,
\end{equation}
where $\Delta B_i = B((i+1)\delta)-B(i\delta)$ and both coefficients are evaluated at the start of the final substep.
If $\tau_N < T$, the scheme executes only the substeps for the drift integrated against $t$ on the terminal interval $[\tau_N, T]$, defining the terminal approximation by $X_{\Delta,\delta}(T) = X_N^{m_N}$.
If $\tau_N=T$, we set $m_N=0$ and $X_N^0=X_{\Delta,\delta}(T)=X_N$.

\begin{remark}\label{rem:stem-design}
	Evaluating $h_\rho, f_\rho, g_\rho$ at the same state $X_i^{m_i-1}$ in the final substep allows their superlinear growth to be balanced simultaneously via the joint coercivity condition in Assumption~\ref{ass:coerc}.
	Throughout the preceding substeps $j = 0, \ldots, m_i-2$, the coercivity bound for $h_\rho$, the truncated coefficient bound, and the substep length control the growth of the stage values.
\end{remark}

\subsection{Moment bounds}\label{subsec:auxiliary-moments}

By Assumption~\ref{ass:space}, the truncated coefficients satisfy the uniform bounds
\begin{equation}\label{eq:coefbound}
	|\varphi_\rho(t, u, x)| \le C \rho^{-\kappa}, \qquad |\varphi_\rho(t, u, x)| \le C(1 + |x|^{\ell+1}),
\end{equation}
for all $(t, u, x) \in [0, T] \times [0, \infty) \times \mathbb R^d$.
Moreover, for these truncated coefficients, \cite{ZuoGuo2026} gives
\begin{equation}\label{eq:truncated-coercivity}
\begin{aligned}
	2\langle x, h_\rho(t, u, x)\rangle &\le C(1 + |x|^2), \\
	2\langle x, f_\rho(t, u, x)\rangle + (P - 1)\|g_\rho(t, u, x)\|^2 &\le C(1 + |x|^2),
\end{aligned}
\end{equation}
where $C$ is independent of $\rho\in(0,1]$, $t$, and $u$.

The following lemma establishes a uniform moment bound for the numerical values $X_i^j$.

\begin{lemma}\label{lem:nummom}
	Under Assumptions~\ref{ass:space} and~\ref{ass:coerc} with $P\ge4$, for every $0 < r < P$, there exists a constant $C > 0$ independent of $\Delta, \delta$ such that
	\begin{equation*}
		\mathbb E_B\left[\max_{0 \le i \le N} \max_{0 \le j \le m_i} (1 + |X_i^j|^2)^{\frac{r}{2}}\right] \le C e^{C(T + E(T))}.
	\end{equation*}
\end{lemma}
\begin{proof}
Set $Y(x) = (1+|x|^2)^{\frac{P}{2}}$.
Its derivatives satisfy $\|\nabla^j Y(x)\|\le C(1+|x|^2)^{\frac{P-j}{2}}$ for $1\le j\le4$.
We first bound the substep values $X_i^j$ on each interval $[\tau_i,\tau_{i+1}\wedge T]$ with $\tau_i<T$.
For each substep $0\le j<m_i$, Taylor expansion of $Y(x)$ combined with~\eqref{eq:driftstages} and~\eqref{eq:truncated-coercivity} yields
\begin{align*}
	Y(X_i^{j+1}) &\le Y(X_i^j) + P \Delta t_i^j (1+|X_i^j|^2)^{\frac{P}{2}-1} \langle X_i^j, h_\rho(t_i^j,i\delta,X_i^j)\rangle + C (\Delta t_i^j)^2 \rho^{-2\kappa} Y(X_i^j) \\
	&\le (1 + C \Delta t_i^j) Y(X_i^j),
\end{align*}
where we have used $(\Delta t_i^j)^2\rho^{-2\kappa}\le \Delta t_i^j \rho^{1-2\kappa}\le \Delta t_i^j$.
Iterating over the substeps and applying the inequality $1+z\le e^z$, we obtain
\begin{equation}\label{eq:drift-stage-bound}
	\max_{0 \le j \le m_i} Y(X_i^j) \le \exp\left(C \sum_{j=0}^{m_i-1} \Delta t_i^j\right) Y(X_i) = e^{C\Delta\tau_i} Y(X_i).
\end{equation}

Next, we bound the endpoint value $X_{i+1}$ for $i<N$.
Combining the final substep in~\eqref{eq:driftstages} with~\eqref{eq:blockupdate} gives
\begin{equation}\label{eq:block-increment}
	X_{i+1} - X_i^{m_i-1} = \Delta t_i^{m_i-1} h_\rho + \delta f_\rho + g_\rho \Delta B_i,
\end{equation}
where all coefficients are evaluated at $(t_i^{m_i-1},i\delta,X_i^{m_i-1})$.
The third-order Taylor expansion of $Y(X_{i+1})$ around $X_i^{m_i-1}$ takes the form
\begin{align}\label{eq:taylor-expansion}
	Y(X_{i+1}) &= Y(X_i^{m_i-1}) + \nabla Y(X_i^{m_i-1})^\top (X_{i+1} - X_i^{m_i-1}) \nonumber\\
	&\quad + \frac{1}{2} (X_{i+1} - X_i^{m_i-1})^\top \nabla^2 Y(X_i^{m_i-1}) (X_{i+1} - X_i^{m_i-1}) \nonumber\\
	&\quad + \frac{1}{6} \nabla^3 Y(X_i^{m_i-1})(X_{i+1} - X_i^{m_i-1})^3 + R_4,
\end{align}
where the fourth-order remainder satisfies
\begin{equation*}
	|R_4| \le C\bigl(1+|X_i^{m_i-1}|^2\bigr)^{\frac{P-4}{2}} |X_{i+1} - X_i^{m_i-1}|^4 + C |X_{i+1} - X_i^{m_i-1}|^P.
\end{equation*}
Substituting~\eqref{eq:block-increment} into~\eqref{eq:taylor-expansion}, taking $\mathbb E_B$ on both sides, and using~\eqref{eq:truncated-coercivity}, we obtain
\begin{align*}
	&\quad \mathbb E_B[Y(X_{i+1})] \\
	&\le \mathbb E_B\Bigl[Y(X_i^{m_i-1}) + P\bigl(1+|X_i^{m_i-1}|^2\bigr)^{\frac{P}{2}-1}\langle X_i^{m_i-1}, \Delta t_i^{m_i-1}h_\rho + \delta f_\rho\rangle \\
	&\qquad + \frac{1}{2}(g_\rho\Delta B_i)^\top \nabla^2 Y(X_i^{m_i-1}) (g_\rho\Delta B_i)\Bigr] + C(\Delta t_i^{m_i-1}+\delta)\mathbb E_B[Y(X_i^{m_i-1})] \\
	&\le \mathbb E_B\Bigl[Y(X_i^{m_i-1}) + \frac{P}{2}\bigl(1+|X_i^{m_i-1}|^2\bigr)^{\frac{P}{2}-1}\Bigl(2\Delta t_i^{m_i-1}\langle X_i^{m_i-1},h_\rho\rangle \\
	&\qquad + \delta\bigl(2\langle X_i^{m_i-1},f_\rho\rangle+(P-1)\|g_\rho\|^2\bigr)\Bigr)\Bigr] + C(\Delta t_i^{m_i-1}+\delta)\mathbb E_B[Y(X_i^{m_i-1})] \\
	&\le (1+C(\Delta t_i^{m_i-1}+\delta))\mathbb E_B[Y(X_i^{m_i-1})],
\end{align*}
where the higher-order terms and remainder are bounded by $C(\Delta t_i^{m_i-1}+\delta)\mathbb E_B[Y(X_i^{m_i-1})]$ using~\eqref{eq:coefbound}, $\Delta t_i^{m_i-1}+\delta\le2\rho$, and $P(1/2-\kappa)\ge1$.
Combining this with~\eqref{eq:drift-stage-bound}, we deduce
\begin{equation}\label{eq:block-recursion}
	\mathbb E_B[Y(X_{i+1})] \le e^{C(\Delta\tau_i+\delta)}\mathbb E_B[Y(X_i)].
\end{equation}
Since the Brownian increment $\Delta B_i$ is independent of the past trajectory up to $i\delta$, \eqref{eq:block-recursion} holds conditionally on the Brownian history up to $i\delta$.
Recalling $A(t) = t + E(t)$ and $E(\tau_i) = i\delta$, we have $A(\tau_{i+1}) - A(\tau_i) = \Delta\tau_i + \delta$.
Hence the discrete process
\begin{equation*}
	M_i = e^{-C A(\tau_i)} Y(X_i), \qquad 0\le i\le N,
\end{equation*}
is a nonnegative supermartingale with respect to the Brownian history up to $i\delta$, starting from $M_0 = Y(x_0)$ and satisfying $A(\tau_i) \le A(T) = T+E(T)$.
Since $Y(X_i) \le e^{C A(T)} M_i$ for all $0 \le i \le N$, applying Doob's maximal inequality to the supermartingale $(M_i)_{0\le i\le N}$ gives
\begin{equation*}
	\mathbb P_B\left(\max_{0\le i\le N} Y(X_i) > v\right) \le \min\left\{1,\frac{e^{C(T+E(T))}Y(x_0)}{v}\right\}, \qquad v>0.
\end{equation*}
Integrating this tail estimate with exponent $r/P \in (0, 1)$ yields
\begin{equation}\label{eq:node-moments}
	\mathbb E_B\left[\max_{0\le i\le N}(1+|X_i|^2)^{\frac{r}{2}}\right] \le \frac{\bigl(e^{C(T+E(T))}Y(x_0)\bigr)^{\frac{r}{P}}}{1-\frac{r}{P}} \le C e^{C(T+E(T))}.
\end{equation}
Finally, raising~\eqref{eq:drift-stage-bound} to the power $r/P$, taking the maximum over $0 \le i \le N$, and using $\Delta\tau_i \le T$, we obtain
\begin{align*}
	\mathbb E_B\left[\max_{0 \le i \le N} \max_{0 \le j \le m_i} (1 + |X_i^j|^2)^{\frac{r}{2}}\right] &\le C \mathbb E_B\left[\max_{0 \le i \le N} (1 + |X_i|^2)^{\frac{r}{2}}\right] \\
	&\le C e^{C(T+E(T))},
\end{align*}
which completes the proof.
\end{proof}

For the convergence analysis, we introduce the continuous-time approximation $\bar{X}(s)$ on $[0,T]$ following~\cite{Mao2015,LiuMaoTangWu2020,ZuoGuo2026}.
For each $0\le i\le N$ with $\tau_i<T$, $0\le j<m_i$, and $s\in[t_i^j,t_i^{j+1}]$, define
\begin{align}\label{eq:aux-def}
	\bar{X}(s) &= X_i^j + (s - t_i^j) h_\rho(t_i^j, i\delta, X_i^j) 
	+ (E(s) - i\delta) f_\rho(t_i^{m_i-1}, i\delta, X_i^{m_i-1}) \nonumber\\
	&\quad+ g_\rho(t_i^{m_i-1}, i\delta, X_i^{m_i-1})\bigl(B(E(s)) - B(i\delta)\bigr),
\end{align}
with $\bar{X}(T)=X_{\Delta,\delta}(T)$ if $T=\tau_N$.
By construction, $\bar{X}$ is a continuous and adapted process satisfying $\bar{X}(0)=x_0$ and $\bar{X}(\tau_i)=X_i$ for $0\le i\le N$.
On the time grid $\mathcal T_\delta=\{\tau_i:0\le i\le N\}\cup\{T\}$, $\bar{X}(t)$ and $X_{\Delta,\delta}(t)$ can differ only at $t=T$ when $\tau_N<T$.

\begin{lemma}\label{lem:nodal-residual}
	Under Assumptions~\ref{ass:space}--\ref{ass:coerc} with $P\ge4$, for every $2\le p<q$ with $p(2\ell+1)<P$,
	\begin{equation*}
		\mathbb E_B\left[\max_{t\in\mathcal T_\delta}|\bar{X}(t)-X_{\Delta,\delta}(t)|^p\right] \le C e^{C(T+E(T))}\delta^{\frac{p}{2}}.
	\end{equation*}
	Consequently,
	\begin{equation*}
		\mathbb E\left[\max_{t\in\mathcal T_\delta}|\bar{X}(t)-X_{\Delta,\delta}(t)|^p\right] \le C\delta^{\frac{p}{2}}.
	\end{equation*}
\end{lemma}
\begin{proof}
	For $\tau_N<T$, since $X_{\Delta,\delta}(T)=X_N^{m_N}=X_N^{m_N-1}+\Delta t_N^{m_N-1}h_\rho(t_N^{m_N-1},N\delta,X_N^{m_N-1})$, evaluating~\eqref{eq:aux-def} at $s=T$ with $j=m_N-1$ yields
	\begin{align*}
		\bar{X}(T)-X_{\Delta,\delta}(T) &= (E(T)-N\delta)f_\rho(t_N^{m_N-1},N\delta,X_N^{m_N-1}) \\
		&\quad + g_\rho(t_N^{m_N-1},N\delta,X_N^{m_N-1})\bigl(B(E(T))-B(N\delta)\bigr).
	\end{align*}
	Since $\bar{X}(\tau_i)=X_{\Delta,\delta}(\tau_i)$ for $0\le i\le N$ and $0\le E(T)-N\delta\le\delta$, taking $\mathbb E_B$ and applying~\eqref{eq:coefbound} together with Lemma~\ref{lem:nummom} for $r=p(\ell+1)<P$, we obtain
	\begin{align*}
		&\quad \mathbb E_B\left[\max_{t\in\mathcal T_\delta}|\bar{X}(t)-X_{\Delta,\delta}(t)|^p\right] \\
		&= \mathbb E_B\left[|\bar{X}(T)-X_{\Delta,\delta}(T)|^p\right] \\
		&\le C\delta^p\mathbb E_B[|f_\rho(t_N^{m_N-1},N\delta,X_N^{m_N-1})|^p]+C\delta^{\frac{p}{2}}\mathbb E_B[\|g_\rho(t_N^{m_N-1},N\delta,X_N^{m_N-1})\|^p] \\
		&\le C\left(\delta^p+\delta^{\frac{p}{2}}\right)\mathbb E_B[1+|X_N^{m_N-1}|^{p(\ell+1)}] \\
		&\le C e^{C(T+E(T))}\delta^{\frac{p}{2}}.
	\end{align*}
	Taking expectation with respect to $D$ and applying Lemma~\ref{lem:inverse-moments} completes the proof.
\end{proof}

\section{Strong convergence rate}\label{sec:convergence}

Throughout this section, we fix an exponent $p$ satisfying
\begin{equation}\label{eq:prange}
	2 \le p < q, \qquad P > \left(1+\frac{1}{\kappa}\right)p(\ell+1).
\end{equation}
Since $\kappa\in(0,1/4]$, condition~\eqref{eq:prange} implies $P\ge4$ and $p(2\ell+1)<P$.

\subsection{Convergence rates for general subordinators}\label{subsec:defects}

For $0\le i\le N$ with $\tau_i<T$, $0\le j<m_i$, and $s\in[t_i^j,t_i^{j+1})$, set
\begin{equation}\label{eq:defects-def}
\begin{aligned}
	\mathcal R_h(s) &= h(s,E(s),\bar{X}(s))-h_\rho(t_i^j,i\delta,X_i^j), \\
	\mathcal R_f(s) &= f(s, E(s), \bar{X}(s)) - f_\rho(t_i^{m_i-1},i\delta,X_i^{m_i-1}), \\
	\mathcal R_g(s) &= g(s, E(s), \bar{X}(s)) - g_\rho(t_i^{m_i-1},i\delta,X_i^{m_i-1}).
\end{aligned}
\end{equation}

With these residuals, $\bar{X}$ satisfies the perturbed equation preceding Lemma~\ref{lem:stability}.
\begin{lemma}\label{lem:defects}
	For every exponent $p$ satisfying~\eqref{eq:prange} and $s\in[\tau_i,\tau_{i+1}\wedge T)$,
	\begin{align*}
		\mathbb E_B\left[|\mathcal R_h(s)|^p\right]^{\frac{2}{p}} &\le C e^{C(T+E(T))} \bigl(\Delta^2 + \delta + \Delta^{2\beta_{h,t}} + \delta^{2\beta_{h,u}}\bigr), \\
		\mathbb E_B\left[|\mathcal R_\varphi(s)|^p\right]^{\frac{2}{p}} &\le C e^{C(T+E(T))} \bigl(\Delta^2+(\Delta\tau_i)^2+\delta+(\Delta\tau_i)^{2\beta_{\varphi,t}}+\delta^{2\beta_{\varphi,u}}\bigr),
	\end{align*}
	where $\varphi\in\{f,g\}$.
\end{lemma}
\begin{proof}
	Fix $s\in[t_i^j,t_i^{j+1})$.
	We first estimate the spatial differences in~\eqref{eq:defects-def}.
	By~\eqref{eq:aux-def}, \eqref{eq:coefbound}, $|s-t_i^j|\le\Delta$, and Lemma~\ref{lem:nummom} with $p(2\ell+1)<P$, we have
	\begin{align*}
		\mathbb E_B\left[\left(1+\max_{0\le k\le m_i}|X_i^k|^\ell\right)^p |\bar{X}(s)-X_i^j|^p\right]^{\frac{1}{p}}
		&\le Ce^{C(T+E(T))}(\Delta+\delta^{\frac{1}{2}}), \\
		\mathbb E_B\left[|\bar{X}(s)-X_i^j|^{p(\ell+1)}\right]^{\frac{1}{p}}
		&\le Ce^{C(T+E(T))}(\Delta+\delta^{\frac{1}{2}}),
	\end{align*}
	where the second bound uses $|g_\rho|\le C\rho^{-\kappa}$ with $\delta\le\rho$ and $\kappa\le 1/4$.
	Moreover, iterating~\eqref{eq:driftstages} and using~\eqref{eq:coefbound} yields
	\begin{align*}
		|X_i^j - X_i^{m_i-1}| &\le \sum_{k=j}^{m_i-2}\Delta t_i^k |h_\rho(t_i^k,i\delta,X_i^k)| \\
		&\le C\Delta\tau_i\left(1+\max_{0\le k\le m_i}|X_i^k|^{\ell+1}\right).
	\end{align*}
	Applying Assumption~\ref{ass:space}, the triangle inequality $|\bar{X}(s)-X_i^{m_i-1}|\le |\bar{X}(s)-X_i^j|+|X_i^j-X_i^{m_i-1}|$, and Lemma~\ref{lem:nummom} now gives
	\begin{equation}\label{eq:spatial-defects}
	\begin{aligned}
		\mathbb E_B\left[|h(s,E(s),\bar{X}(s))-h(s,E(s),X_i^j)|^p\right]^{\frac{1}{p}} &\le Ce^{C(T+E(T))}(\Delta+\delta^{\frac{1}{2}}), \\
		\mathbb E_B\left[|\varphi(s,E(s),\bar{X}(s))-\varphi(s,E(s),X_i^{m_i-1})|^p\right]^{\frac{1}{p}} &\le Ce^{C(T+E(T))}(\Delta+\Delta\tau_i+\delta^{\frac{1}{2}}),
	\end{aligned}
	\end{equation}
	for $\varphi\in\{f,g\}$.
	By Assumption~\ref{ass:time}, $|s-t_i^j|\le\Delta$, and $|s-t_i^{m_i-1}|\le \Delta\tau_i$, the temporal differences satisfy
	\begin{equation}\label{eq:temporal-defects}
	\begin{aligned}
		\mathbb E_B\left[|h(s,E(s),X_i^j)-h(t_i^j,i\delta,X_i^j)|^p\right]^{\frac{1}{p}} &\le Ce^{C(T+E(T))}\bigl(\Delta^{\beta_{h,t}}+\delta^{\beta_{h,u}}\bigr), \\
		\mathbb E_B\left[|\varphi(s,E(s),X_i^{m_i-1})-\varphi(t_i^{m_i-1},i\delta,X_i^{m_i-1})|^p\right]^{\frac{1}{p}} &\le Ce^{C(T+E(T))}\bigl((\Delta\tau_i)^{\beta_{\varphi,t}}+\delta^{\beta_{\varphi,u}}\bigr),
	\end{aligned}
	\end{equation}
	for $\varphi\in\{f,g\}$.
	Finally, fix $r_* \in [(1+1/\kappa)p(\ell+1), P)$.
	For $\varphi\in\{h,f,g\}$ and any stage value $X_i^j$, since $\varphi(t,u,X_i^j)=\varphi_\rho(t,u,X_i^j)$ for $|X_i^j|\le R$, using~\eqref{eq:coefbound} and Lemma~\ref{lem:nummom} yields
	\begin{align}
		\mathbb E_B\left[|\varphi(t,u,X_i^j)-\varphi_\rho(t,u,X_i^j)|^p\right]^{\frac{1}{p}}
		&\le C R^{-\frac{r_*-p(\ell+1)}{p}}\mathbb E_B[|X_i^j|^{r_*}]^{\frac{1}{p}} \nonumber\\
		&\le C e^{C(T+E(T))}\rho^{\frac{\kappa(r_*-p(\ell+1))}{p(\ell+1)}} \nonumber\\
		&\le C e^{C(T+E(T))}(\Delta+\delta^{\frac{1}{2}}). \label{eq:trunc-defects}
	\end{align}
	Combining~\eqref{eq:spatial-defects}, \eqref{eq:temporal-defects}, and~\eqref{eq:trunc-defects} by Minkowski's inequality and squaring the $L^p$ bounds completes the proof.
\end{proof}

\begin{lemma}\label{lem:integrated}
	For $p$ in~\eqref{eq:prange}, the residual integral in~\eqref{eq:forcing} satisfies
	\begin{equation*}
		\mathcal I(T) \le C e^{C(T+E(T))}(1+E(T)) \Bigl(\Delta^{2\beta_{h,t}} + \delta^{\min\{1, 2\beta_{f,t}, 2\beta_{g,t}, 2\beta_u\}}\Bigr).
	\end{equation*}
\end{lemma}
\begin{proof}
	Since $\int_{\tau_i}^{\tau_{i+1}\wedge T}\mathrm dE(s) = E(\tau_{i+1}\wedge T) - i\delta \le \delta$, integrating the bounds in Lemma~\ref{lem:defects} against $\mathrm ds$ and $\mathrm dE(s)$ yields
	\begin{align}
		\mathcal I(T) &\le C e^{C(T+E(T))} \biggl( \int_0^T \bigl(\Delta^2 + \delta + \Delta^{2\beta_{h,t}} + \delta^{2\beta_{h,u}}\bigr) \mathrm ds \nonumber\\
		&\qquad + \sum_{i=0}^N \int_{\tau_i}^{\tau_{i+1}\wedge T} \bigl(\Delta^2 + \delta + \delta^{2\beta_u} + (\Delta\tau_i)^2 + (\Delta\tau_i)^{2\beta_{f,t}} + (\Delta\tau_i)^{2\beta_{g,t}}\bigr) \mathrm dE(s) \biggr) \nonumber\\
		&\le C e^{C(T+E(T))} \biggl(\Delta^{2\beta_{h,t}} + (1+E(T))\bigl(\Delta^2+\delta+\delta^{2\beta_u}\bigr) \nonumber\\
		&\qquad + \delta\sum_{i=0}^N \bigl((\Delta\tau_i)^2 + (\Delta\tau_i)^{2\beta_{f,t}} + (\Delta\tau_i)^{2\beta_{g,t}}\bigr)\biggr). \label{eq:forcing-blocks}
	\end{align}
	Since $\sum_{i=0}^N \Delta\tau_i = T$ and $N \le E(T)/\delta$, for any $a > 0$ we have
	\begin{equation}\label{eq:blocksum}
		\delta \sum_{i=0}^N (\Delta\tau_i)^a \le C \delta^{\min\{1, a\}} (1 + E(T))^{(1-a) \vee 0}.
	\end{equation}
	Applying~\eqref{eq:blocksum} with $a\in\{2, 2\beta_{f,t}, 2\beta_{g,t}\}$ to~\eqref{eq:forcing-blocks} and using $\Delta^2\le\Delta^{2\beta_{h,t}}$ completes the proof.
\end{proof}

We now establish the main strong convergence result for general subordinators.

\begin{theorem}\label{thm:main}
	Under Assumptions~\ref{ass:space}--\ref{ass:coerc}, let $p$ satisfy~\eqref{eq:prange}, and choose the discretization and truncation parameters as in Section~\ref{subsec:stem-scheme}.
	Then
	\begin{align*}
		&\quad \mathbb E\left[\sup_{0\le t\le T}|X(t)-\bar{X}(t)|^p\right] + \mathbb E\left[\max_{t\in\mathcal T_\delta}|X(t)-X_{\Delta,\delta}(t)|^p\right] \\
		&\le C\Bigl(\Delta^{2\beta_{h,t}}+\delta^{\min\{1, 2\beta_{f,t}, 2\beta_{g,t}, 2\beta_u\}}\Bigr)^{\frac{p}{2}}.
	\end{align*}
\end{theorem}
\begin{proof}
	Combining Lemma~\ref{lem:stability} with Lemma~\ref{lem:integrated}, taking expectation with respect to $D$, and applying Lemma~\ref{lem:inverse-moments}, we obtain
	\begin{align*}
		\mathbb E\left[\sup_{0\le t\le T}|X(t)-\bar{X}(t)|^p\right]
		&\le C \mathbb E\left[e^{C(T+E(T))}\mathcal I(T)^{\frac{p}{2}}\right] \\
		&\le C \mathbb E\left[e^{C(T+E(T))}(1+E(T))^{\frac{p}{2}}\right]\Bigl(\Delta^{2\beta_{h,t}}+\delta^{\min\{1, 2\beta_{f,t}, 2\beta_{g,t}, 2\beta_u\}}\Bigr)^{\frac{p}{2}} \\
		&\le C\Bigl(\Delta^{2\beta_{h,t}}+\delta^{\min\{1, 2\beta_{f,t}, 2\beta_{g,t}, 2\beta_u\}}\Bigr)^{\frac{p}{2}}.
	\end{align*}
	Together with Lemma~\ref{lem:nodal-residual}, the triangle inequality on $\mathcal T_\delta$ then gives
	\begin{align*}
		\mathbb E\left[\max_{t\in\mathcal T_\delta}|X(t)-X_{\Delta,\delta}(t)|^p\right]
		&\le C \mathbb E\left[\sup_{0\le t\le T}|X(t)-\bar{X}(t)|^p\right] + C \mathbb E\left[\max_{t\in\mathcal T_\delta}|\bar{X}(t)-X_{\Delta,\delta}(t)|^p\right] \\
		&\le C\Bigl(\Delta^{2\beta_{h,t}}+\delta^{\min\{1, 2\beta_{f,t}, 2\beta_{g,t}, 2\beta_u\}}\Bigr)^{\frac{p}{2}},
	\end{align*}
	which completes the proof.
\end{proof}

For equal step sizes $\delta = \Delta$, Theorem~\ref{thm:main} implies that the strong convergence order in $L^p$ is $\min\{\beta_t, \beta_u, 1/2\}$.
In particular, when $\beta_t, \beta_u \ge 1/2$, the SSTEM scheme achieves the strong convergence order $1/2$.

\subsection{\texorpdfstring{Sharpened rates under a Laplace exponent bound}{Sharpened rates under a Laplace exponent bound}}\label{subsec:theorems}
The following growth condition on the Laplace exponent $\Phi$ yields a sharper moment bound for sums of powers of the random interval lengths.
\begin{assumption}\label{ass:laplace-growth}
There exist ${\gamma}\in(0,1)$ and $L_5>0$ such that
\begin{equation*}
	\Phi(\lambda)\le L_5\lambda^{\gamma}, \qquad \lambda\ge1.
\end{equation*}
\end{assumption}
For convenience, we define $\omega_a\colon (0,1]\to(0,\infty)$ for each $a>0$ by
\begin{equation*}
	\omega_a(\delta) = \begin{cases} 
		\delta^{\frac{a}{{\gamma}} \wedge 1}, & a \ne {\gamma}, \\ 
		\delta(1 + |\log\delta|), & a = {\gamma}. 
	\end{cases}
\end{equation*}

\begin{lemma}\label{lem:stable-blocks}
	Suppose Assumption~\ref{ass:laplace-growth} holds, and fix $a>0$, $r\ge1$, and $c>0$.
	Then, uniformly for all $\delta \in (0, 1]$,
	\begin{equation*}
		\mathbb E\left[e^{c E(T)} \left(\delta\sum_{i=0}^N (\Delta\tau_i)^a\right)^r\right] \le C \omega_a(\delta)^r.
	\end{equation*}
\end{lemma}
\begin{proof}
Recall that $\tau_i=D(i\delta)$ for all integers $i\ge0$, so $\tau_{i+1}-\tau_i$ has the same law as $D(\delta)$ by stationary increments.
For $0<x\le T$, Assumption~\ref{ass:laplace-growth} and the monotonicity of $\Phi$ give $\Phi(1/x)\le Cx^{-{\gamma}}$.
Applying Markov's inequality and using the Laplace transform, we obtain
\begin{align*}
	\mathbb P(D(\delta)>x) &= \mathbb P\left(1-e^{-D(\delta)/x}>1-e^{-1}\right) \\
	&\le \frac{\mathbb E\left[1-e^{-D(\delta)/x}\right]}{1-e^{-1}} \\
	&= \frac{1-e^{-\delta\Phi(1/x)}}{1-e^{-1}} \\
	&\le C\min\{1,\delta x^{-{\gamma}}\}.
\end{align*}
Consequently, by Fubini's theorem, for every $b > 0$,
\begin{align}
	\mathbb E\left[(D(\delta)\wedge T)^b\right] 
	&= b\int_0^T x^{b-1}\mathbb P(D(\delta)>x) \mathrm dx \nonumber\\
	&\le C b\int_0^T x^{b-1} \min\{1, \delta x^{-{\gamma}}\} \mathrm dx \nonumber\\
	&= C b \left(\int_0^{T\wedge\delta^{1/{\gamma}}} x^{b-1} \mathrm{d}x + \delta \int_{T\wedge\delta^{1/{\gamma}}}^T x^{b-{\gamma}-1} \mathrm{d}x\right) \nonumber\\
	&\le C \omega_b(\delta), \label{eq:single-inc-moment}
\end{align}
where evaluating the two integrals for $b \ne {\gamma}$ and $b = {\gamma}$ yields the definition of $\omega_b(\delta)$.

To control the sum for a deterministic number $n \ge 1$ of terms, observe that for every integer $j \ge 1$,
\begin{equation*}
	\delta^{j-1}\omega_{aj}(\delta) \le C \omega_a(\delta)^j.
\end{equation*}
For an integer $m\ge2r$, the multinomial expansion, independence of the increments, and~\eqref{eq:single-inc-moment} with $b=ar_j$ give
\begin{align}
	&\quad \mathbb E\left[\left(\delta\sum_{i=0}^{n-1}\bigl((\tau_{i+1}-\tau_i)\wedge T\bigr)^a\right)^m\right] \nonumber\\
	&\le C\delta^m\sum_{k=1}^m n^k \sum_{\substack{r_1+\cdots+r_k=m\\r_1,\ldots,r_k\ge1}} \prod_{j=1}^k \omega_{ar_j}(\delta) \nonumber\\
	&\le C \omega_a(\delta)^m\sum_{k=1}^m(n\delta)^k \nonumber\\
	&\le C(1+n\delta)^m\omega_a(\delta)^m. \label{eq:determ-block-sum}
\end{align}
The multinomial coefficients depend only on $m$ and are absorbed into $C$.

To handle the random index $N$ and the factor $e^{cE(T)}$, we partition the sample space into the events $\{k\le E(T)<k+1\}$ for integers $k\ge0$.
Since $N\delta\le E(T)<k+1$ on each such event, we have $N\le\lceil(k+1)/\delta\rceil-1$, and hence
\begin{align*}
	e^{cE(T)} &\le e^{c(k+1)}, \\
	\delta\sum_{i=0}^{N}(\Delta\tau_i)^a &\le \delta\sum_{i=0}^{\lceil(k+1)/\delta\rceil-1}\bigl((\tau_{i+1}-\tau_i)\wedge T\bigr)^a.
\end{align*}
Applying the Cauchy--Schwarz inequality, followed by Lyapunov's inequality with $m\ge2r$ and~\eqref{eq:determ-block-sum}, gives
\begin{align*}
	&\quad \mathbb E\left[\mathbf 1_{\{k\le E(T)<k+1\}}e^{cE(T)}\left(\delta\sum_{i=0}^{N}(\Delta\tau_i)^a\right)^r\right] \\
	&\le e^{c(k+1)}\left(\mathbb E\left[\left(\delta\sum_{i=0}^{\lceil(k+1)/\delta\rceil-1}\bigl((\tau_{i+1}-\tau_i)\wedge T\bigr)^a\right)^{2r}\right]\right)^{\frac12}\mathbb P(E(T)\ge k)^{\frac12} \\
	&\le Ce^{c(k+1)}(k+3)^r\omega_a(\delta)^r\mathbb P(E(T)\ge k)^{\frac12},
\end{align*}
where $\lceil(k+1)/\delta\rceil\delta\le k+2$ for $\delta\in(0,1]$.
By the inverse relation and Markov's inequality, for every $\lambda>0$,
\begin{align*}
	\mathbb P(E(T)\ge k) &\le \mathbb P(D(k)\le T) \\
	&=\mathbb P\left(e^{-\lambda D(k)}\ge e^{-\lambda T}\right) \\
	&\le e^{\lambda T}\mathbb E\left[e^{-\lambda D(k)}\right] \\
	&=e^{\lambda T-k\Phi(\lambda)}.
\end{align*}
Summing over $k\ge0$ and combining the preceding bounds, we obtain
\begin{align*}
	&\quad \mathbb E\left[e^{cE(T)}\left(\delta\sum_{i=0}^{N}(\Delta\tau_i)^a\right)^r\right] \\*
	&=\sum_{k=0}^{\infty}\mathbb E\left[\mathbf 1_{\{k\le E(T)<k+1\}}e^{cE(T)}\left(\delta\sum_{i=0}^{N}(\Delta\tau_i)^a\right)^r\right] \\
	&\le C\omega_a(\delta)^r e^{c+\frac{\lambda T}{2}}\sum_{k=0}^{\infty}(k+3)^r e^{-k(\frac{\Phi(\lambda)}{2}-c)} \\
	&\le C\omega_a(\delta)^r.
\end{align*}
Since $\Phi(\lambda)\to\infty$ as $\lambda\to\infty$, we can choose $\lambda$ such that $\Phi(\lambda)>2c$, so the series converges independently of $\delta$, which completes the proof.
\end{proof}

\begin{theorem}\label{thm:stable-rate}
	Under the assumptions of Theorem~\ref{thm:main}, suppose further that Assumption~\ref{ass:laplace-growth} holds.
	If $2\beta_{f,t} \ne {\gamma}$ and $2\beta_{g,t} \ne {\gamma}$, then
	\begin{align*}
		&\quad \mathbb E\left[\sup_{0\le t\le T}|X(t)-\bar{X}(t)|^p\right]+\mathbb E\left[\max_{t\in\mathcal T_\delta}|X(t)-X_{\Delta,\delta}(t)|^p\right] \\
		&\le C\Bigl(\Delta^{2\beta_{h,t}}+\delta^{\min\left\{1, \frac{2\beta_{f,t}}{{\gamma}}, \frac{2\beta_{g,t}}{{\gamma}}, 2\beta_u\right\}}\Bigr)^{\frac{p}{2}}.
	\end{align*}
	If $2\beta_{f,t} = {\gamma}$ or $2\beta_{g,t} = {\gamma}$, the same bound holds with an additional factor $1 + |\log\delta|$ multiplying $\delta^{\min\left\{1, \frac{2\beta_{f,t}}{{\gamma}}, \frac{2\beta_{g,t}}{{\gamma}}, 2\beta_u\right\}}$.
\end{theorem}
\begin{proof}
\rev{Applying Lemma~\ref{lem:stable-blocks} to the sums in~\eqref{eq:forcing-blocks} in place of~\eqref{eq:blocksum} and proceeding as in the proof of Theorem~\ref{thm:main} yields the stated bounds, which completes the proof.}
\end{proof}

\begin{remark}\label{rem:laplace-examples}
Assumption~\ref{ass:laplace-growth} includes stable subordinators with $\Phi(\lambda)=c\lambda^\alpha$ and tempered stable subordinators with $\Phi(\lambda)=c((\lambda+\eta)^\alpha-\eta^\alpha)$, where $0<\alpha<1$ and $c,\eta>0$.
For both classes, one may take $\gamma=\alpha$ in Assumption~\ref{ass:laplace-growth}.
For a Gamma subordinator, $\Phi(\lambda)=c\log(1+\lambda/\eta)$ satisfies Assumption~\ref{ass:laplace-growth} for arbitrarily small $\gamma>0$.
Theorem~\ref{thm:stable-rate} then yields an $L^p$ strong error bound of order
$\Delta^{\beta_{h,t}}+\delta^{\min\{1/2,\beta_u\}}.$
\end{remark}

\section{Numerical experiments and applications in solving FPEs}\label{sec:numerical}

In this section, two examples are given. Example \ref{ex:chirp} displays the strong convergece rate of the proposed method. Example  \ref{ex:transport-pde} provides a case that the propsed method can be used to numerically approximate some FPEs.

We examine the strong convergence of SSTEM and compare it with the $D$-grid truncated Euler--Maruyama (DTEM) scheme and the equidistant truncated Euler--Maruyama (ETEM) scheme.
The comparison uses the same driving paths and observation nodes for all three methods at each step size.

\begin{example}\label{ex:chirp}
We consider the scalar SDE
\begin{align*}
\mathrm{d}X(t)&=\left(\left(0.20+0.03\sin(2\pi t)+0.03\frac{E(t)}{1+E(t)}\right)X(t)-X^3(t)\right)\mathrm{d}t\\
&\quad+\left(\left(0.15+0.02\cos(2\pi t)+0.02\frac{E(t)}{1+E(t)}\right)X(t)-X^5(t)\right)\mathrm{d}E(t)\\
&\quad+\left(0.25X^2(t)+0.05\frac{E(t)}{1+E(t)}+0.5\psi(t)\right)\mathrm{d}B(E(t)),
\end{align*}
with $X(0)=0.9$ and $T=2$, where
\begin{equation*}
\psi(t)=
\begin{cases}
8|t-0.5|^3\sin\bigl(|t-0.5|^{-9}\bigr), & t\ne0.5,\\
0, & t=0.5.
\end{cases}
\end{equation*}
Here $E$ is the inverse of an $\alpha$-stable subordinator with Laplace exponent $\Phi(\lambda)=\lambda^\alpha$, independent of the Brownian motion $B$.
The perturbation $\psi$ has H\"older exponent $\beta=\beta_{g,t}=0.3$.
With $\beta_{h,t}=\beta_{f,t}=\beta_u=1$ and $\ell=4$, Theorem~\ref{thm:main} gives an $L^2$ strong error bound of order $h^{0.3}$ when $\Delta=\delta=h$.
Theorem~\ref{thm:stable-rate}, applied with $\gamma=\alpha$, sharpens this bound to
\begin{equation*}
\begin{cases}
h^{1/2}, & 0<\alpha<0.6,\\
h^{1/2}\sqrt{1+|\log h|}, & \alpha=0.6,\\
h^{0.3/\alpha}, & 0.6<\alpha<1.
\end{cases}
\end{equation*}
\end{example}

For each $\alpha$, we use $M=2000$ independent Monte Carlo samples.
All methods, tested step sizes, and the reference approximation share the same fine samples of $D$ and $B$.
We set
\begin{equation*}
\delta=\Delta=h_j=2^{-j}, \qquad j=9,\ldots,15,
\end{equation*}
and use $R_0=2$, $\kappa=1/6$, and $\ell=4$, so that the truncation radius is $R_h=2h^{-1/30}$.
The truncated coefficients $h_h,f_h,g_h$ are those defined in Section~\ref{subsec:stem-scheme} with $\rho=h$.
The reference approximation $X_{\rm ref}$ is computed by SSTEM with
\begin{equation*}
\delta_{\rm ref}=2^{-20}, \qquad \Delta_{\rm ref}=2^{-18}, \qquad R_{\rm ref}=2\bigl(\max\{\Delta_{\rm ref},\delta_{\rm ref}\}\bigr)^{-1/30}.
\end{equation*}

DTEM advances on the sampled $D$-grid without substeps for the drift integrated against $t$.
ETEM uses the equidistant grid $t_k=kh$.
For each path $i$ and step size $h$, all three methods are evaluated at
\begin{equation*}
\mathcal T_h^{(i)}=\{D^{(i)}(nh):n\ge0,\ D^{(i)}(nh)\le T\}\cup\{T\}.
\end{equation*}
Writing $Y_h$ for the approximation, we compute the root-mean-square error (RMSE)
\begin{equation}\label{eq:empirical-rmse}
e(h)=\left(\frac1M\sum_{i=1}^M\max_{t\in\mathcal T_h^{(i)}}\left|Y_h^{(i)}(t)-X_{\rm ref}^{(i)}(t)\right|^2\right)^{1/2}.
\end{equation}
For ETEM, $Y_h(t)$ denotes its approximation at the observation time $t$.

For each method, the experimental order of convergence (EOC) is defined by
\begin{equation*}
\mathrm{EOC}_j=\frac{\log e(h_j)-\log e(h_{j-1})}{\log h_j-\log h_{j-1}}, \qquad j=10,\ldots,15.
\end{equation*}
Since $h_j=h_{j-1}/2$, this equals $\log_2(e(h_{j-1})/e(h_j))$.

Table~\ref{tab:chirp-rates} summarizes the SSTEM rates for six values of $\alpha$.
The numerical convergence rates are close to $1/2$ for $\alpha<0.6$ and to $0.3/\alpha$ for $\alpha>0.6$.
At $\alpha=0.6$, the numerical convergence rate is $0.4653$, while the theoretical estimate includes the factor $\sqrt{1+|\log h|}$.
Therefore, the error bound in this case includes a logarithmic factor.
\begin{table}[htbp]
\centering
\caption{Theoretical and numerical convergence rates of SSTEM for Example~\ref{ex:chirp}.}
\label{tab:chirp-rates}
\small
\setlength{\tabcolsep}{6pt}
\begin{tabular*}{\linewidth}{@{\extracolsep{\fill}}ccccccc@{}}
\toprule
$\alpha$ & $0.3$ & $0.4$ & $0.5$ & $0.6$ & $0.7$ & $0.8$\\
\midrule
General bound & $0.3000$ & $0.3000$ & $0.3000$ & $0.3000$ & $0.3000$ & $0.3000$\\
Sharpened bound & $0.5000$ & $0.5000$ & $0.5000$ & $0.5000$ & $0.4286$ & $0.3750$\\
SSTEM numerical rate & $0.4987$ & $0.4964$ & $0.4940$ & $0.4653$ & $0.4322$ & $0.3760$\\
\bottomrule
\end{tabular*}
\end{table}

Tables~\ref{tab:chirp-comparison06} and~\ref{tab:chirp-comparison08} report the errors and EOCs for the three methods.
For both values of $\alpha$, SSTEM has the smallest errors at all tested step sizes, with numerical convergence rates of $0.4653$ and $0.3760$ for $\alpha=0.6$ and $\alpha=0.8$, respectively.
For ETEM, the numerical convergence rates are $0.3346$ and $0.2594$, close to the temporal H\"older exponent $\beta=0.3$ considered in~\cite{ZuoGuo2026,ZuoGuo2026TCSDDE}.
For DTEM, the EOCs at the finest step size are $0.001$ for $\alpha=0.6$ and $0.012$ for $\alpha=0.8$, suggesting a lack of convergence consistent with the difficulty noted in~\cite{JinKobayashi2021}.

\begin{table}[htbp]
\centering
\caption{RMSE and EOC for Example~\ref{ex:chirp} with $\alpha=0.6$.}
\label{tab:chirp-comparison06}
\small
\setlength{\tabcolsep}{6pt}
\begin{tabular*}{\linewidth}{@{\extracolsep{\fill}}lrrrrrr@{}}
\toprule
& \multicolumn{2}{c}{SSTEM} & \multicolumn{2}{c}{DTEM} & \multicolumn{2}{c}{ETEM}\\
\cmidrule(lr){2-3}\cmidrule(lr){4-5}\cmidrule(lr){6-7}
$h$ & Error & EOC & Error & EOC & Error & EOC\\
\midrule
$2^{-9}$ & $0.1634$ & -- & $0.3441$ & -- & $0.5193$ & --\\
$2^{-10}$ & $0.1225$ & $0.415$ & $0.3259$ & $0.079$ & $0.3987$ & $0.381$\\
$2^{-11}$ & $0.08824$ & $0.474$ & $0.3192$ & $0.030$ & $0.3188$ & $0.322$\\
$2^{-12}$ & $0.06520$ & $0.437$ & $0.3120$ & $0.033$ & $0.2740$ & $0.219$\\
$2^{-13}$ & $0.04648$ & $0.488$ & $0.3115$ & $0.002$ & $0.2159$ & $0.344$\\
$2^{-14}$ & $0.03347$ & $0.474$ & $0.3094$ & $0.010$ & $0.1658$ & $0.380$\\
$2^{-15}$ & $0.02368$ & $0.499$ & $0.3092$ & $0.001$ & $0.1218$ & $0.445$\\
\midrule
Numerical rate & \multicolumn{2}{c}{$0.4653$} & \multicolumn{2}{c}{$0.0231$} & \multicolumn{2}{c}{$0.3346$}\\
\bottomrule
\end{tabular*}
\end{table}

\begin{table}[htbp]
\centering
\caption{RMSE and EOC for Example~\ref{ex:chirp} with $\alpha=0.8$.}
\label{tab:chirp-comparison08}
\small
\setlength{\tabcolsep}{6pt}
\begin{tabular*}{\linewidth}{@{\extracolsep{\fill}}lrrrrrr@{}}
\toprule
& \multicolumn{2}{c}{SSTEM} & \multicolumn{2}{c}{DTEM} & \multicolumn{2}{c}{ETEM}\\
\cmidrule(lr){2-3}\cmidrule(lr){4-5}\cmidrule(lr){6-7}
$h$ & Error & EOC & Error & EOC & Error & EOC\\
\midrule
$2^{-9}$  & $0.2700$  & --      & $0.3660$ & --      & $0.3612$ & --\\
$2^{-10}$ & $0.2145$  & $0.332$ & $0.3328$ & $0.137$ & $0.3132$ & $0.206$\\
$2^{-11}$ & $0.1697$  & $0.338$ & $0.3049$ & $0.126$ & $0.2645$ & $0.244$\\
$2^{-12}$ & $0.1304$  & $0.381$ & $0.2884$ & $0.080$ & $0.2326$ & $0.186$\\
$2^{-13}$ & $0.09944$ & $0.391$ & $0.2740$ & $0.074$ & $0.1909$ & $0.285$\\
$2^{-14}$ & $0.07482$ & $0.410$ & $0.2704$ & $0.019$ & $0.1524$ & $0.325$\\
$2^{-15}$ & $0.05720$ & $0.387$ & $0.2681$ & $0.012$ & $0.1215$ & $0.326$\\
\midrule
Numerical rate & \multicolumn{2}{c}{$0.3760$} & \multicolumn{2}{c}{$0.0750$} & \multicolumn{2}{c}{$0.2594$}\\
\bottomrule
\end{tabular*}
\end{table}

\begin{example}\label{ex:transport-pde}
We consider the two-dimensional \rev{SDE}
\begin{equation*}
\left\{
\begin{aligned}
\mathrm{d}X_1(t)&=\bigl(X_1(t)-X_1^3(t)+0.5X_2(t)\bigr)\mathrm{d}t-\frac{X_1^3(t)}{1+X_1^2(t)}\mathrm{d}E(t)+0.8\sqrt{1+X_1^2(t)}\mathrm{d}B_1(E(t)),\\
\mathrm{d}X_2(t)&=\bigl(X_2(t)-X_2^3(t)+0.5X_1(t)\bigr)\mathrm{d}t-\frac{X_2^3(t)}{1+X_2^2(t)}\mathrm{d}E(t)+0.8\sqrt{1+X_2^2(t)}\mathrm{d}B_2(E(t)).
\end{aligned}
\right.
\end{equation*}
with $X(0)=(0.5,0.5)^\top$.
Here $E$ is the inverse of a standard $\alpha$-stable subordinator, and $B_1,B_2$ are independent standard Brownian motions, independent of $E$.
\end{example}

Following the generalized FPE framework in~\cite{EuleFriedrichJenkoSokolov2008}, we use the memory equation
\begin{equation*}
\partial_t\mu_t=\mathcal A\mu_t+\mathcal B\mathcal M_{\mathcal A}^{1-\alpha}\mu(t), \qquad \mu_0=\delta_{(0.5,0.5)},
\end{equation*}
for the law $\mu_t=\operatorname{Law}(X(t))$, where
\begin{align*}
\mathcal A\mu&=-\sum_{i=1}^2\partial_{x_i}\bigl((x_i-x_i^3+0.5x_{3-i})\mu\bigr),\\
\mathcal B\mu&=\sum_{i=1}^2\partial_{x_i}\left(\frac{x_i^3}{1+x_i^2}\mu\right)+0.32\sum_{i=1}^2\partial_{x_ix_i}^2\bigl((1+x_i^2)\mu\bigr),\\
\mathcal M_{\mathcal A}^{1-\alpha}\mu(t)&=(\partial_t-\mathcal A)\frac{1}{\Gamma(\alpha)}\int_0^t(t-s)^{\alpha-1}\mathcal S_{t-s}\mu_s\mathrm{d}s.
\end{align*}
Here $\mathcal S_r$ is the transport semigroup generated by $\mathcal A$, and the equation holds in the sense of distributions.

We take $\alpha=0.8$ and $T=1$, and simulate $10^5$ SSTEM paths with $\Delta=\delta=2^{-12}$ and radial truncation radius $R=4\delta^{-1/24}$.
The drift integrated against $t$ is approximated using explicit Euler substeps of length at most $\Delta$.
The final substep of each completed $D$-grid interval also includes the drift and Brownian increments associated with $\delta$, with all coefficients evaluated at the same truncated left-endpoint state.
The independent PDE computation uses $1024$ time steps and a $192\times192$ spatial grid on $[-4,4]^2$ in the coordinates $u_i=\operatorname{arsinh}(x_i)$.

Figures~\ref{fig:transport-pde08} and~\ref{fig:transport-marginal08} compare the joint and marginal probability density functions (PDFs) at $t=1$, respectively.
The joint densities exhibit similar concentration patterns, while the marginal histograms closely follow the PDE curves.
Visible differences remain near the joint density peaks.

On the same machine using a single CPU thread, SSTEM required $134.12$ seconds, including the generation of the subordinator $D$, while the PDE computation required $261.24$ seconds.
When the sampled values of $D$ were precomputed and stored, the SSTEM computation took $23.77$ seconds.
These timings are reported without optimization of either implementation.

\begin{figure}[!htbp]
\centering
\includegraphics[width=.48\linewidth]{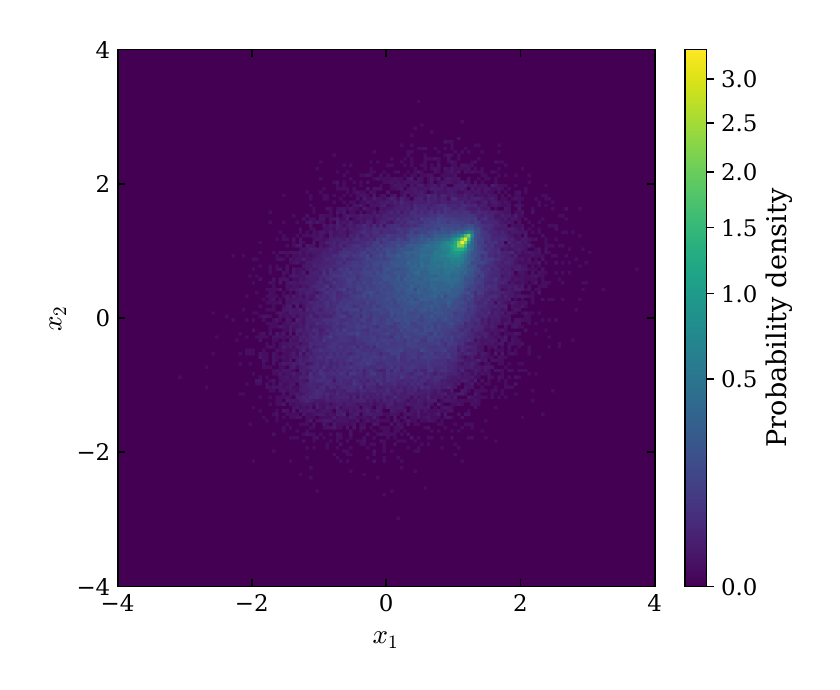}\hfill
\includegraphics[width=.48\linewidth]{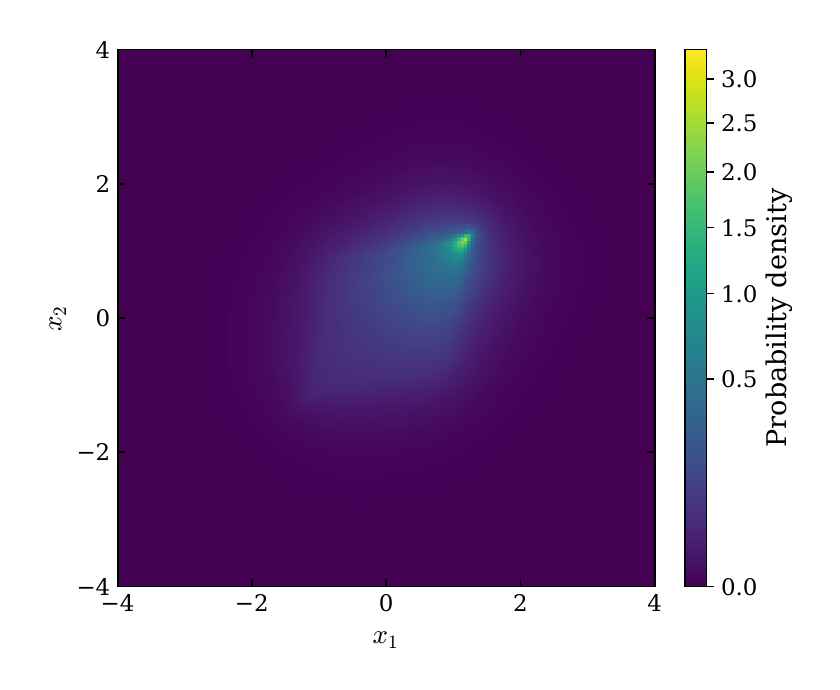}
\caption{Joint PDFs for Example~\ref{ex:transport-pde} with $\alpha=0.8$ at $t=1$.
Left: SSTEM.
Right: numerical FPE solution.
}\label{fig:transport-pde08}
\end{figure}

\begin{figure}[!htbp]
\centering
\includegraphics[width=.48\linewidth]{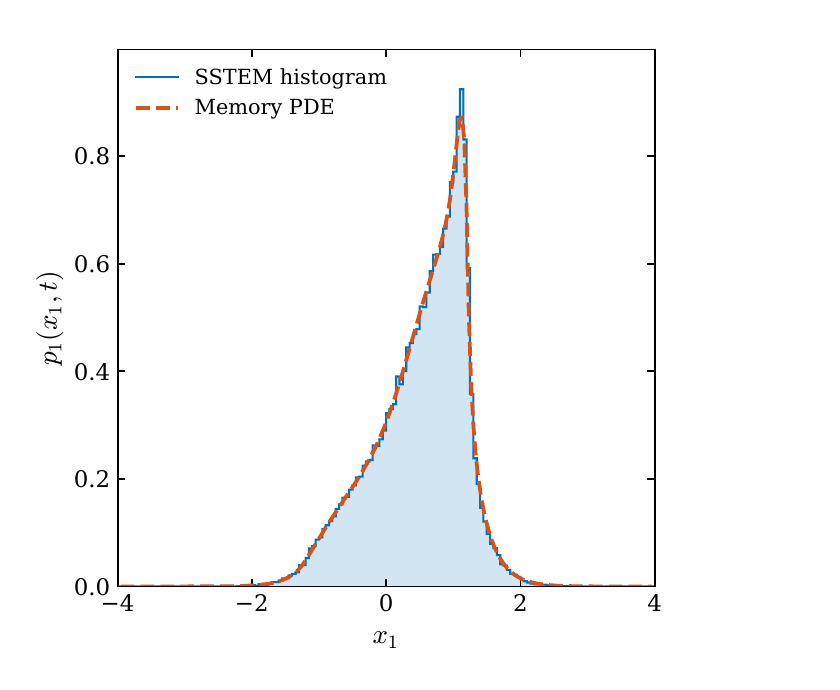}\hfill
\includegraphics[width=.48\linewidth]{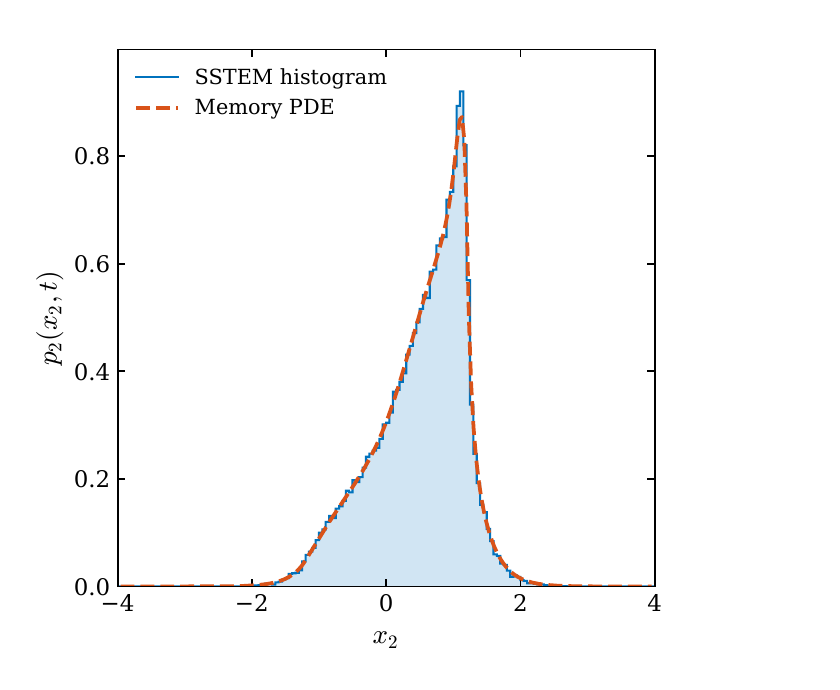}
\caption{Marginal PDFs of $X_1$ (left) and $X_2$ (right) for Example~\ref{ex:transport-pde} with $\alpha=0.8$ at $t=1$.
	The histograms and curves represent the SSTEM and numerical FPE solutions, respectively.
}\label{fig:transport-marginal08}
\end{figure}

\section{Conclusion}\label{sec:conclusion}

In this paper, we developed the SSTEM scheme for superlinear SDEs driven by subdiffusion with a state-dependent drift integrated against $t$.
The scheme combines sampled subordinator values with local sub-stepping for the drift integrated against $t$.
In the final substep of each completed subordinator interval, $h_\rho$, $f_\rho$, and $g_\rho$ are evaluated at the same state to balance their superlinear growth.
Under joint monotonicity and finite-order coercivity conditions, we established step-independent moment bounds and strong error estimates for independently chosen time and subordinator sampling steps, with sharpened rates under a power growth bound on the Laplace exponent.
The scalar example illustrated the effect of the subordinator index on the convergence rate.
For a coupled two-dimensional model, SSTEM simulations were compared with a numerical FPE solution.
Future work may investigate whether the sharpened strong convergence rates obtained under low temporal regularity are optimal.
Further directions include weak convergence analysis and the use of SSTEM to study the associated FPE.

\section*{Statements and Declarations}

\noindent\textbf{Funding.} Qian Guo and Wei Liu were supported by the National Natural Science Foundation of China under Grant No.~12271368.
\par\noindent\textbf{Competing interests.} The authors have no relevant financial or non-financial interests to disclose.
\par\noindent\textbf{Data availability.} The data generated during the current study are available from the first author upon reasonable request.
\par\noindent\textbf{Declaration of AI use.} The authors used ChatGPT to assist with English language editing and improving readability.
All suggested changes were reviewed and verified by the authors.
The authors take full responsibility for the content of this manuscript.

\begingroup
\small
\bibliographystyle{plain}
\bibliography{references}

@article{HighamMaoStuart2002,
  author = {Higham, Desmond J. and Mao, Xuerong and Stuart, Andrew M.},
  title = {Strong convergence of {Euler}-type methods for nonlinear stochastic differential equations},
  journal = {SIAM J. Numer. Anal.},
  year = {2002},
  volume = {40},
  number = {3},
  pages = {1041--1063},
  doi = {10.1137/S0036142901389530},
  note = {\url{https://doi.org/10.1137/S0036142901389530}}
}

@article{MeerschaertScheffler2004,
  author = {Meerschaert, Mark M. and Scheffler, Hans-Peter},
  title = {Limit theorems for continuous-time random walks with infinite mean waiting times},
  journal = {J. Appl. Probab.},
  year = {2004},
  volume = {41},
  number = {3},
  pages = {623--638},
  doi = {10.1239/jap/1091543414},
  note = {\url{https://doi.org/10.1239/jap/1091543414}}
}

@article{HutzenthalerJentzenKloeden2011,
  author = {Hutzenthaler, Martin and Jentzen, Arnulf and Kloeden, Peter E.},
  title = {Strong and weak divergence in finite time of {Euler}'s method for stochastic differential equations with non-globally {Lipschitz} continuous coefficients},
  journal = {Proc. R. Soc. A},
  year = {2011},
  volume = {467},
  number = {2130},
  pages = {1563--1576},
  doi = {10.1098/rspa.2010.0348},
  note = {\url{https://doi.org/10.1098/rspa.2010.0348}}
}

@article{Kobayashi2011,
  author = {Kobayashi, Kei},
  title = {Stochastic calculus for a time-changed semimartingale and the associated stochastic differential equations},
  journal = {J. Theor. Probab.},
  year = {2011},
  volume = {24},
  number = {3},
  pages = {789--820},
  doi = {10.1007/s10959-010-0320-9},
  note = {\url{https://doi.org/10.1007/s10959-010-0320-9}}
}

@article{Mao2015,
  author = {Mao, Xuerong},
  title = {The truncated {Euler--Maruyama} method for stochastic differential equations},
  journal = {J. Comput. Appl. Math.},
  year = {2015},
  volume = {290},
  pages = {370--384},
  doi = {10.1016/j.cam.2015.06.002},
  note = {\url{https://doi.org/10.1016/j.cam.2015.06.002}}
}

@article{JumKobayashi2016,
  author = {Jum, Ernest and Kobayashi, Kei},
  title = {A strong and weak approximation scheme for stochastic differential equations driven by a time-changed {Brownian} motion},
  journal = {Probab. Math. Statist.},
  year = {2016},
  volume = {36},
  number = {2},
  pages = {201--220},
  note = {\url{https://www.math.uni.wroc.pl/~pms/files/36.2/Article/36.2.2.pdf}}
}

@article{JinKobayashi2019,
  author = {Jin, Sixian and Kobayashi, Kei},
  title = {Strong approximation of stochastic differential equations driven by a time-changed {Brownian} motion with time-space-dependent coefficients},
  journal = {J. Math. Anal. Appl.},
  year = {2019},
  volume = {476},
  number = {2},
  pages = {619--636},
  doi = {10.1016/j.jmaa.2019.04.001},
  note = {\url{https://doi.org/10.1016/j.jmaa.2019.04.001}}
}

@article{DengLiu2020,
  author = {Deng, Chang-Song and Liu, Wei},
  title = {Semi-implicit {Euler--Maruyama} method for non-linear time-changed stochastic differential equations},
  journal = {BIT Numer. Math.},
  year = {2020},
  volume = {60},
  number = {4},
  pages = {1133--1151},
  doi = {10.1007/s10543-020-00810-7},
  note = {\url{https://doi.org/10.1007/s10543-020-00810-7}}
}

@article{LiuMaoTangWu2020,
  author = {Liu, Wei and Mao, Xuerong and Tang, Jingwen and Wu, Yue},
  title = {Truncated {Euler--Maruyama} method for classical and time-changed non-autonomous stochastic differential equations},
  journal = {Appl. Numer. Math.},
  year = {2020},
  volume = {153},
  pages = {66--81},
  doi = {10.1016/j.apnum.2020.02.007},
  note = {\url{https://doi.org/10.1016/j.apnum.2020.02.007}}
}

@article{JinKobayashi2021,
  author = {Jin, Sixian and Kobayashi, Kei},
  title = {Strong approximation of time-changed stochastic differential equations involving drifts with random and non-random integrators},
  journal = {BIT Numer. Math.},
  year = {2021},
  volume = {61},
  number = {3},
  pages = {829--857},
  doi = {10.1007/s10543-021-00852-5},
  note = {\url{https://doi.org/10.1007/s10543-021-00852-5}}
}

@article{LiLiuTang2025,
  author = {Li, Xiaotong and Liu, Wei and Tang, Tianjiao},
  title = {Truncated {Euler--Maruyama} method for time-changed stochastic differential equations with super-linear state variables and {H\"older}'s continuous time variables},
  journal = {J. Comput. Math.},
  year = {2025},
  volume = {43},
  number = {5},
  pages = {1194--1218},
  doi = {10.4208/jcm.2411-m2022-0061},
  note = {\url{https://doi.org/10.4208/jcm.2411-m2022-0061}}
}

@article{LiuWangZuo2025,
  author = {Liu, Wei and Wang, Meng and Zuo, Ruchun},
  title = {Parameter-related strong convergence rate of the backward {Euler--Maruyama} method for time-changed stochastic differential equations},
  journal = {Fluct. Noise Lett.},
  year = {2025},
  volume = {24},
  number = {3},
  pages = {2550036},
  doi = {10.1142/S0219477525500361},
  note = {\url{https://doi.org/10.1142/S0219477525500361}}
}

@article{Zuo2026,
  author = {Zuo, Ruchun},
  title = {Parameter-related strong convergence rates of {Euler}-type methods for time-changed stochastic differential equations},
  journal = {Numer. Algorithms},
  year = {2026},
  doi = {10.1007/s11075-026-02454-0},
  note = {\url{https://doi.org/10.1007/s11075-026-02454-0}}
}

@article{ZuoGuo2026,
  author = {Zuo, Ruchun and Guo, Qian},
  title = {An equidistant truncated {Euler--Maruyama} method for nonlinear time-changed stochastic differential equations},
  journal = {Research Square},
  year = {2026},
  doi = {10.21203/rs.3.rs-10974346/v1},
  note = {\url{https://doi.org/10.21203/rs.3.rs-10974346/v1}}
}

@misc{ZuoGuo2026TCSDDE,
  author = {Zuo, Ruchun and Guo, Qian},
  title = {Strong convergence of fully discrete tamed {Euler} approximations for time-changed stochastic delay differential equations},
  year = {2026},
  howpublished = {arXiv:2610.28552},
  doi = {10.48550/arXiv.2610.28552},
  note = {\url{https://doi.org/10.48550/arXiv.2610.28552}}
}

@article{MetzlerKlafter2000,
  author = {Metzler, Ralf and Klafter, Joseph},
  title = {The random walk's guide to anomalous diffusion: a fractional dynamics approach},
  journal = {Phys. Rep.},
  year = {2000},
  volume = {339},
  number = {1},
  pages = {1--77},
  doi = {10.1016/S0370-1573(00)00070-3},
  note = {\url{https://doi.org/10.1016/S0370-1573(00)00070-3}}
}

@article{MeerschaertStraka2013,
  author = {Meerschaert, Mark M. and Straka, Peter},
  title = {Inverse stable subordinators},
  journal = {Math. Model. Nat. Phenom.},
  year = {2013},
  volume = {8},
  number = {2},
  pages = {1--16},
  doi = {10.1051/mmnp/20138201},
  note = {\url{https://doi.org/10.1051/mmnp/20138201}}
}

@article{MagdziarzWeronWeron2007,
  author = {Magdziarz, Marcin and Weron, Aleksander and Weron, Karina},
  title = {Fractional {Fokker--Planck} dynamics: stochastic representation and computer simulation},
  journal = {Phys. Rev. E},
  year = {2007},
  volume = {75},
  number = {1},
  pages = {016708},
  doi = {10.1103/PhysRevE.75.016708},
  note = {\url{https://doi.org/10.1103/PhysRevE.75.016708}}
}

@article{Magdziarz2009Langevin,
  author = {Magdziarz, Marcin},
  title = {{Langevin} picture of subdiffusion with infinitely divisible waiting times},
  journal = {J. Stat. Phys.},
  year = {2009},
  volume = {135},
  number = {4},
  pages = {763--772},
  doi = {10.1007/s10955-009-9751-z},
  note = {\url{https://doi.org/10.1007/s10955-009-9751-z}}
}

@article{ChenWangDeng2019,
  author = {Chen, Yao and Wang, Xudong and Deng, Weihua},
  title = {Subdiffusion in an external force field},
  journal = {Phys. Rev. E},
  year = {2019},
  volume = {99},
  number = {4},
  pages = {042125},
  doi = {10.1103/PhysRevE.99.042125},
  note = {\url{https://doi.org/10.1103/PhysRevE.99.042125}}
}

@article{CarnaffanKawai2017,
  author = {Carnaffan, Sean and Kawai, Reiichiro},
  title = {Solving multidimensional fractional {Fokker--Planck} equations via unbiased density formulas for anomalous diffusion processes},
  journal = {SIAM J. Sci. Comput.},
  year = {2017},
  volume = {39},
  number = {5},
  pages = {B886--B915},
  doi = {10.1137/17M111482X},
  note = {\url{https://doi.org/10.1137/17M111482X}}
}

@article{WuZuo2026,
  author = {Wu, Ruoxue and Zuo, Ruchun},
  title = {A {Milstein}-type method for highly non-linear non-autonomous time-changed stochastic differential equations},
  journal = {Appl. Numer. Math.},
  year = {2026},
  volume = {230},
  pages = {209--226},
  doi = {10.1016/j.apnum.2026.07.013},
  note = {\url{https://doi.org/10.1016/j.apnum.2026.07.013}}
}

@article{GrabchakSaba2025,
  author = {Grabchak, Michael and Saba, Sina},
  title = {On approximations of subordinators in {$L^p$} and the simulation of tempered stable distributions},
  journal = {Stat. Comput.},
  year = {2025},
  volume = {35},
  pages = {54},
  doi = {10.1007/s11222-025-10586-x},
  note = {\url{https://doi.org/10.1007/s11222-025-10586-x}}
}

@article{HutzenthalerJentzenKloeden2012,
  author = {Hutzenthaler, Martin and Jentzen, Arnulf and Kloeden, Peter E.},
  title = {Strong convergence of an explicit numerical method for {SDEs} with nonglobally {Lipschitz} continuous coefficients},
  journal = {Ann. Appl. Probab.},
  year = {2012},
  volume = {22},
  number = {4},
  pages = {1611--1641},
  doi = {10.1214/11-AAP803},
  note = {\url{https://doi.org/10.1214/11-AAP803}}
}

@article{Sabanis2016,
  author = {Sabanis, Sotirios},
  title = {{Euler} approximations with varying coefficients: the case of superlinearly growing diffusion coefficients},
  journal = {Ann. Appl. Probab.},
  year = {2016},
  volume = {26},
  number = {4},
  pages = {2083--2105},
  doi = {10.1214/15-AAP1140},
  note = {\url{https://doi.org/10.1214/15-AAP1140}}
}

@article{Mao2016,
  author = {Mao, Xuerong},
  title = {Convergence rates of the truncated {Euler--Maruyama} method for stochastic differential equations},
  journal = {J. Comput. Appl. Math.},
  year = {2016},
  volume = {296},
  pages = {362--375},
  doi = {10.1016/j.cam.2015.09.035},
  note = {\url{https://doi.org/10.1016/j.cam.2015.09.035}}
}

@article{LiMaoYin2019,
  author = {Li, Xiaoyue and Mao, Xuerong and Yin, George},
  title = {Explicit numerical approximations for stochastic differential equations in finite and infinite horizons: truncation methods, convergence in {$p$}th moment, and stability},
  journal = {IMA J. Numer. Anal.},
  year = {2019},
  volume = {39},
  number = {2},
  pages = {847--892},
  doi = {10.1093/imanum/dry015},
  note = {\url{https://doi.org/10.1093/imanum/dry015}}
}

@article{Lageras2005,
  author = {Lager{\aa}s, Andreas Nordvall},
  title = {A renewal-process-type expression for the moments of inverse subordinators},
  journal = {J. Appl. Probab.},
  year = {2005},
  volume = {42},
  number = {4},
  pages = {1134--1144},
  doi = {10.1239/jap/1134587822},
  note = {\url{https://doi.org/10.1239/jap/1134587822}}
}

@article{MagdziarzSchilling2015,
  author = {Magdziarz, Marcin and Schilling, Ren\'e L.},
  title = {Asymptotic properties of {Brownian} motion delayed by inverse subordinators},
  journal = {Proc. Am. Math. Soc.},
  year = {2015},
  volume = {143},
  number = {10},
  pages = {4485--4501},
  doi = {10.1090/proc/12588},
  note = {\url{https://doi.org/10.1090/proc/12588}}
}

@article{KumarVellaisamy2015,
  author = {Kumar, A. and Vellaisamy, P.},
  title = {Inverse tempered stable subordinators},
  journal = {Stat. Probab. Lett.},
  year = {2015},
  volume = {103},
  pages = {134--141},
  doi = {10.1016/j.spl.2015.04.010},
  note = {\url{https://doi.org/10.1016/j.spl.2015.04.010}}
}

@misc{WangNiuZhang2026,
  author = {Wang, Shuai and Niu, Yuanling and Zhang, Ying},
  title = {{$1/2$} order convergence rate of {Euler}-type methods for time-changed stochastic differential equations with super-linearly growing drift and diffusion coefficients},
  year = {2026},
  howpublished = {arXiv:2507.14562v4},
  doi = {10.48550/arXiv.2507.14562},
  note = {\url{https://arxiv.org/abs/2507.14562v4}}
}

@article{ColantoniDOvidio2023,
  author = {Colantoni, Fausto and D'Ovidio, Mirko},
  title = {On the inverse gamma subordinator},
  journal = {Stoch. Anal. Appl.},
  year = {2023},
  volume = {41},
  number = {5},
  pages = {999--1024},
  doi = {10.1080/07362994.2022.2108450},
  note = {\url{https://doi.org/10.1080/07362994.2022.2108450}}
}

@article{Magdziarz2009TimeDependent,
  author = {Magdziarz, Marcin},
  title = {Stochastic representation of subdiffusion processes with time-dependent drift},
  journal = {Stochastic Process. Appl.},
  year = {2009},
  volume = {119},
  number = {10},
  pages = {3238--3252},
  doi = {10.1016/j.spa.2009.05.006},
  note = {\url{https://doi.org/10.1016/j.spa.2009.05.006}}
}

@article{ChenLiuWu2026,
  author = {Chen, Ziheng and Liu, Jiao and Wu, Anxin},
  title = {Strong and weak convergence orders of numerical methods for {SDEs} driven by time-changed {L\'evy} noise},
  journal = {Appl. Numer. Math.},
  year = {2026},
  volume = {221},
  pages = {46--62},
  doi = {10.1016/j.apnum.2025.11.007},
  note = {\url{https://doi.org/10.1016/j.apnum.2025.11.007}}
}

@article{WenLiXu2023,
  author = {Wen, Xueqi and Li, Zhi and Xu, Liping},
  title = {Strong approximation of non-autonomous time-changed {McKean--Vlasov} stochastic differential equations},
  journal = {Commun. Nonlinear Sci. Numer. Simul.},
  year = {2023},
  volume = {119},
  pages = {107122},
  doi = {10.1016/j.cnsns.2023.107122},
  note = {\url{https://doi.org/10.1016/j.cnsns.2023.107122}}
}

@article{EuleFriedrich2009,
  author = {Eule, S. and Friedrich, R.},
  title = {Subordinated {Langevin} equations for anomalous diffusion in external potentials---biasing and decoupled external forces},
  journal = {EPL},
  year = {2009},
  volume = {86},
  number = {3},
  pages = {30008},
  doi = {10.1209/0295-5075/86/30008},
  note = {\url{https://doi.org/10.1209/0295-5075/86/30008}}
}

@article{CairoliBaule2015,
  author = {Cairoli, Andrea and Baule, Adrian},
  title = {{Langevin} formulation of a subdiffusive continuous time random walk in physical time},
  journal = {Phys. Rev. E},
  year = {2015},
  volume = {92},
  number = {1},
  pages = {012102},
  doi = {10.1103/PhysRevE.92.012102},
  note = {\url{https://doi.org/10.1103/PhysRevE.92.012102}}
}

@misc{zhang2025subdiffusive,
  author = {Zhang, Shuaiqi and Chen, Zhen-Qing},
  title = {Sub-diffusive {Black--Scholes} model and {Girsanov} transform for sub-diffusions},
  year = {2025},
  eprint = {2511.10371},
  archivePrefix = {arXiv},
  primaryClass = {math.PR},
  howpublished = {arXiv:2511.10371},
  url = {https://arxiv.org/abs/2511.10371},
  note = {\url{https://arxiv.org/abs/2511.10371}}
}

@article{MagdziarzWeronKlafter2008,
  author = {Magdziarz, Marcin and Weron, Aleksander and Klafter, Joseph},
  title = {Equivalence of the fractional {Fokker--Planck} and subordinated {Langevin} equations: the case of a time-dependent force},
  journal = {Phys. Rev. Lett.},
  year = {2008},
  volume = {101},
  number = {21},
  pages = {210601},
  doi = {10.1103/PhysRevLett.101.210601},
  note = {\url{https://doi.org/10.1103/PhysRevLett.101.210601}}
}

@article{EuleFriedrichJenkoSokolov2008,
  author = {Eule, S. and Friedrich, R. and Jenko, F. and Sokolov, I. M.},
  title = {Continuous-time random walks with internal dynamics and subdiffusive reaction-diffusion equations},
  journal = {Phys. Rev. E},
  year = {2008},
  volume = {78},
  number = {6},
  pages = {060102},
  doi = {10.1103/PhysRevE.78.060102},
  note = {\url{https://doi.org/10.1103/PhysRevE.78.060102}}
}
\endgroup
	
\end{document}